\documentclass[11pt]{amsart} 

\usepackage{graphicx,color}
\usepackage{latexsym}
\usepackage{graphicx} 
\usepackage{amsthm,amsmath, amssymb,amsfonts,esint} 
\usepackage{layout,textcomp}
\usepackage{enumerate}

\usepackage{hyperref}
\usepackage[english]{babel}
\theoremstyle{plain}

\newcommand{\Div}{\mbox{div}}
\newcommand{\tr}{\mbox{tr}}
\newcommand{\Ric}{\mbox{Ric}}

 \newcommand{\Hess}{\mbox{Hess}}

\newtheorem{theorem}{Theorem}[section]

\newtheorem{lemma}[theorem]{Lemma}

\newtheorem{theoremL}{Theorem}

\newtheorem{corollary}[theorem]{Corollary}

\newtheorem{example}[theorem]{Example}

\newtheorem{proposition}[theorem]{Proposition}

\theoremstyle{remark}
\newtheorem{remark}[theorem]{Remark}

\numberwithin{equation}{section}

\begin{document}

\title[Critical metrics and curvature of metrics with constraints]{Critical metrics and curvature of metrics with unit volume or unit area of the boundary}

\author{Tiarlos Cruz }
\address{Universidade Federal de Alagoas\\
Instituto de Matemática\\ 57072-970,\linebreak
Maceió - AL, Brazil}
\email{cicero.cruz@im.ufal.br}
\thanks{Research supported in part by CNPq (311803/2019-9) and FAPITEC/SE/Brazil.}
\author{Almir Silva Santos}
\address{Universidade Federal de Sergipe\\
Departamento de Matemática\\
49100-000, São Cristóvão - SE, Brazil}
\email{almir@mat.ufs.br}

\begin{abstract}
Given a smooth compact manifold with boundary, we study variational properties of the volume functional and of the area functional of the boundary, restricted to the space of the Riemannian metrics with prescribed curvature. We obtain a sufficient and necessary condition for a metric to be a critical point.  As a by-product, a very natural analogue of V-statics metrics is obtained. In the second part, using the Yamabe invariant in the boundary setting, we solve the Kazdan-Warner-Kobayashi problem in a compact manifold with boundary. For several cases, depending on the signal of the Yamabe invariant, we give sufficient and necessary condition for a smooth function to be the scalar or mean curvature of a metric with constraint on the volume or area of the boundary. 
\end{abstract}



\maketitle

%\tableofcontents

\section{Introduction}

Two of the most remarkable topics in differential geometry  are   pheno-mena involving scalar curvature and volume comparison results. Although the interplay  between them, it is well known, we highlight that any relation they  may have, requires a finer analysis. 
Among them, the most fundamental is that the stationary points of the total scalar curvature functional on the space of unit volume  metrics are the Einstein metrics, which  are metrics of constant scalar curvature. 
Considering a  modified problem,  P. Miao and  L.-F. Tam \cite{MT} have studied variational properties for the volume functional on the space of metrics whose scalar curvature is equal to a given constant (see also \cite{MT2}), which is important to locally achieve simultaneously small prescribed perturbation of the scalar curvature and of the volume by a compactly supported deformation of the metric due to a result of  J. Corvino, M. Eichmair and P. Miao \cite{CEM}.

In this paper,  from a variational point of view,  we  investigate the influence of the volume and the area of the boundary   on the most basic geometric invariants in a manifold,
 namely the scalar curvature and the mean curvature of the boundary.

In \cite{CV},  in order to prescribe the curvature of manifolds with boundary, the first named author and F. Vitório proved that the map $g\mapsto (R_g,2H_g)$ is  almost always a surjection, where $R_g$ and $H_g$ denote the  scalar curvature  and  the mean curvature of the boundary of the metric $g$, respectively. In this work we  study  this same map under the effect of the volume and the area of the boundary through two approaches:  on the one hand,  we  vary the volume or area of the boundary and, on the other hand, we restrict it to metrics of unit volume  or unit  area of the boundary. This  second plays a fundamental role  to the problem of prescribing  curvature, the so-called modified Kazdan-Warner-Kobayashi problem.
 
 In the first part of this work, given constants $\kappa$ and $\tau$, we consider the  map $\Psi(g)=(R_g, 2 H_{g}, -2\kappa\mbox{Vol}(g),-2\tau\mbox{Area}(g))$ defined on the space of Riemannian metrics of a smooth compact  manifold with non empty boundary. Denoting by $\mathcal{S}^*_g$ the $L^2$-formal adjoint of the linearization of $\Psi,$ we observe that a triple $(V,\kappa,\tau)$ in the kernel of $\mathcal{S}^*_g$ gives  a  natural analogue for V-static equations. More precisely, a Riemannian manifold $(M,g)$ with non empty boundary $\Sigma$ will satisfies this \textit{V-static-type equation} if there is a nontrivial smooth solution $V$ of the equation
\begin{equation}\label{firststatic}
\left\{\begin{array}{rcll}
     -(\Delta_g V)g+\mbox{Hess}_gV-V\Ric_g & = & \kappa g & \mbox{ on } M\\
     \vspace{-0,2cm}\\
     \displaystyle\frac{\partial V}{\partial \nu}g-V\Pi_{g} & = & \tau g & \mbox{ in }\Sigma,
\end{array}\right.
\end{equation}
where  $\kappa,\tau\in\mathbb R$. Here $\Delta_g $, $ \mbox{Hess}_g$ and $ \Ric_g$  are the Laplacian, the Hessian and  the Ricci
curvature of $ g,$ respectively, and $\Pi_{g}$ is the second fundamental form of the boundary with respect to the outward unit normal $\nu$. This system is invariant under constant rescaling of the metric.

A metric satisfying \eqref{firststatic}  can also be seen as critical point of a Riemannian functional (see Proposition \ref{properties}). Besides, as in \cite{MT}, we give  sufficient and necessary conditions for a metric to be a critical point of the functionals $g\mapsto \mbox{Area($\Sigma$,$g$)}$ and $
g\mapsto \mbox{Vol($M$,$g$)}$. 
Before we state our first theorem, we introduce the following sets.
Let $\mathcal{M}^{k,2}$ be the space of $W^{k,2}$ metrics on $M$, where $k$ is a positive integer. Consider the spaces
$$\mathcal{M}_{0,c}=\left\{g\in \mathcal{M}^{k,2}; R_g=0\;\mbox{ and }\;H_{g}=c\right\}$$
and $$\mathcal{M}_{c,0}=\left\{g\in \mathcal{M}^{k,2}; R_g=c\;\mbox{ and }\;H_{g}=0\right\},$$
where $c$ is constant. 
These spaces are called of \textit{manifolds of Riemannian metrics with prescribed curvature}, provided  they have locally a manifold structure, see  Propositions \ref{subman} and  \ref{subman2}. 
We investigate the problem of finding stationary points for the area functional restricted to $\mathcal{M}_{0,c}$ and the volume functional  restricted to $\mathcal{M}_{c,0}$ (Theorem \ref{generalteo1}).

\begin{theoremL}\label{teo1}
Let $(M,g)$ be a compact Riemannian manifold with non emp-ty boundary $\Sigma$. Let $c\in\mathbb R$ be a fixed constant.
\begin{enumerate}[(a)]

\item  Suppose  $g\in \mathcal{M}_{0,c}$ is a metric such that the first eigenvalue of the Steklov boundary operator $\partial/\partial \nu-c/(n-1)$ is positive. Then 
$g$ is a critical point of the area functional $
g\mapsto \mbox{Area($\Sigma$,$g$)}
$ in $ \mathcal{M}_{0,c}$
if and only if there exists $V\in C^\infty(M)$ such that
\begin{equation*}
\left\{
  \begin{array}{rccl}
-(\Delta_g V)g+\mbox{Hess}_gV-V\Ric_g& = & 0 & \mbox{in } M,\\
\vspace{-0,2cm}\\
\displaystyle\frac{\partial V}{\partial \nu}g-V\Pi_{g} & = & g & \mbox{on }\Sigma.
\end{array}
  \right.
 \end{equation*} 
\item Suppose  $g\in \mathcal{M}_{c,0}$ is a metric such that the first eigenvalue of the ope-rator $(n-1)\Delta_g + c$ with  Neumann boundary condition is positive. Then 
$g$ is a critical point of the volume functional $
g\mapsto \mbox{Vol(M,g)}
$
on $ \mathcal{M}_{c,0}$
if and only if there exists $V\in C^\infty(M)$ such that
\begin{equation*}
\left\{
  \begin{array}{rccl}
-(\Delta_g V)g+\mbox{Hess}_gV-V\Ric_g & = & g & \mbox{in } M,\\
\vspace{-0,2cm}\\
\displaystyle\frac{\partial V}{\partial \nu}g-V\Pi_{g} & = & 0 &\mbox{on } \Sigma.
\end{array}
  \right.
 \end{equation*} 
\end{enumerate}

\end{theoremL}

Let us give another motivation of Theorem \ref{teo1}.
Inspired by the Yamabe problem for manifolds without boundary, J. F. Escobar \cite{E3} introduced the so-called \textit{total scalar curvature plus total mean curvature functional} given by
\begin{equation}
\label{functional} \displaystyle{F(g)=\frac{n-2}{4(n-1)}\int_{M}R_{g}dv+\frac{n-2}{2}\int_{\partial M}H_{g}da}
\end{equation} 
defined on the space of metrics in a manifold with boundary and with dimension $n\geq 3$.
%If $n=2,$  $F(g)=\int_MK+\int_{\partial M}\kappa$ is the total constant equal to $4\pi\chi(M),$ where $\chi(M)$ denotes the Euler characteristic of $M$. 
 The critical points of \eqref{functional} on the space of Riemannian metrics satisfying certain volume and area constraints in a  manifold  with boundary are  precisely the Einstein metrics with umbilical boundary, cf. \cite{A} (compare also Proposition \ref{Einstein} and \ref{Umbilical}).  Theorem \ref{teo1} can be seen as a \textit{counterpart} for this variational characterization.
 
Now consider  a modified action given by the quotient of \eqref{functional}  divided by a normalization term, depending on the volume or the area of the boundary. First, remember that a conformal class $C$ of a manifold $M$ is a collection of metrics such that any two metrics $g_1$ and $g_2$ in $C$ differ by a positive function, i.e, $g_2=e^fg_1$, for some $f\in C^\infty(M)$. Now, for a fixed conformal class $C$, we define the Yamabe constant of $(M,C)$ by
\begin{eqnarray}\label{eq014}
    Y_\lambda(M,C):=\inf_{g\in C}\frac{F(g)}{N_\lambda(g)},
\end{eqnarray}
where $F(g)$ is given by \eqref{functional} for $n\geq 3$,  $\lambda\in\{0,1\}$ and 
\begin{equation*}\label{eq016}
    N_\lambda(g)=\lambda\mbox{Vol}(M,g)^{\frac{n-2}{n}}+(1-\lambda)\mbox{Area}(\partial M,g)^{\frac{n-2}{n-1}}.
\end{equation*}
It is a well known fact that $-\infty\leq Y_\lambda(M,C)\leq Y_\lambda(\mathbb S^n_+,C_0)$, where $C_0$ is the conformal class of the standard metric on the hemisphere. The constants that appears in \eqref{functional} are normalizations, thus by Gauss-Bonnet Theorem, we consider the Yamabe constant and the Yamabe invariant for surface given by
$$\sigma_\lambda(M)=Y_\lambda(M,C)=2\pi\chi(M),$$
while the Yamabe invariant, which can be seen as a natural differential-topological invariant, of a smooth manifold $M$ with non empty boundary and dimension $n\geq 3$ is defined by
\begin{equation}\label{eq015}
    \sigma_\lambda(M)=\sup_CY_\lambda(M,C).
\end{equation}

We make use of the Yamabe invariant and of the variational characte-rization of the map $g\mapsto (R_g,2H_g)$ to prescribe the curvature of manifolds with boundary under some constraints on the volume or boundary area (for  cases without any  constraint see for instance \cite{CV,E5,LR,LMR}). This problem was first approached by O. Kobayashi \cite{K} in the closed case, which was based in the  Kazdan-Warner problem \cite{KW3,KW2} of finding metrics with prescribed scalar curvature. More precisely, Kobayashi  solved in most cases the so-called  modified prescribed scalar curvature problem for metrics with unit volume. The remaining cases turn out to be more difficult and it had been solved by S. Matsuo \cite{M}, thereby completing the solution of the problem. The key ingredient in Matsuo's proof was  a fundamental gluing result for constant scalar curvature metrics proved by J. Corvino, M. Eichmair and P. Miao in \cite{CEM}. His proof uses strongly that a manifold contains a non V-static domain.

In the two-dimensional case, if $\partial M$ has unit Length (resp. $M$ has unit area), and $M$ is a bounded domain $\Omega$ in $\mathbb R^2$ with smooth boundary (resp. compact surface with geodesic boundary), the geodesic curvature $k$ (resp. Gauss curvature $K$)  satisfies one of the following conditions:

\medskip

\begin{enumerate}[(I)]
\item\label{item011}  $k =2\pi\chi(M)$ (resp. $K=2\pi\chi(M)$);
\item\label{item012} $\min k<2\pi\chi(M)<\max k$ (resp. $\min K<2\pi\chi(M)<\max K$).
\end{enumerate}

\medskip

In fact, we prove that \eqref{item011} and \eqref{item012} are  necessary and sufficient conditions to the modified Kazdan-Warner-Kobayashi problem for surfaces. Thus, solving  completely this question  in dimension $n=2.$
\begin{theoremL}\label{dim}
 Let $M$ be a smooth compact surface with non empty boundary. 
 \begin{enumerate}[(a)]
 \item\label{item016} Assume that $\partial M$ is geodesic.  A function $f\in C^\infty(M)$ is the Gauss curvature of some metric $g$ on $M$ with geodesic boundary and $\mbox{Area}(M,g)=1$  if and only if it  satisfies either the condition \eqref{item011} or \eqref{item012}.
 \item\label{item018} Assume that  $M\subset\mathbb{R}^2$ .   A function $f\in C^\infty(\partial M)$ is the geodesic curvature of a flat metric $g$ on $M$ with  $\mbox{Length}(\partial M,g)=1$ if and only if it satisfies either the condition  \eqref{item011} or \eqref{item012}.
 \end{enumerate}
\end{theoremL}

Next we   establish  analogous results in higher dimensions in the case where the  Yamabe invariant is non-positive. One recall that a smooth $n$-dimensional manifold $M$ with non empty boundary has   $\sigma_\lambda(M)\leq 0$ if and only if $M$ does not admit a metric of positive scalar curvature and minimal boundary, for $\lambda=1$, and with scalar curvature equal to zero and positive mean curvature on the boundary, for $\lambda=0$. We should remark that there are results giving topological obstruction to the existence of metrics in ma-nifolds with boundary satisfying conditions on the curvature. For instance, a version of the positive mass theorem \cite{ALB} can be used to prove that does not exist a metric on $(\mathbb{T}
^{n-1}\times[0, 1])\#M_0$ with nonnegative scalar curvature and nonnegative mean curvature along the boundary for $3\leq n\leq 7,$ provided $M_0$  is non-flat, see \cite{BC,Ch}.

\begin{theoremL}\label{prescribing thm}
Let $M^n$ be a compact manifold with non empty boundary $\Sigma$, dimension $n\geq 3$ and $\sigma_1(M)\leq 0.$ 
\begin{enumerate}[(a)]
    \item \label{item002} If $\sigma_1(M)> Y_1(M,C)$ for any conformal class $C,$ then  $f\in C^\infty( M)$  is the scalar curvature of some metric $g$ on $M$ with minimal boundary and $\mbox{Vol}(M,g)=1$  if  and only if  
\begin{equation}\label{eq004}
    \min f<\frac{4(n-1)}{n-2}\sigma_1(M).
\end{equation}

\item\label{item003} If $\sigma_1(M)= Y_1(M,C)$ for some conformal class $C,$ then  $f\in C^\infty( M)$ is the scalar curvature of some metric $g$ on $M$ with minimal boundary and $\mbox{Vol}(M,g)=1$  if and only if  either $f$ is a constant equal to $\frac{4(n-1)}{n-2}\sigma_1(M)$ or \eqref{eq004} holds. \end{enumerate}
Moreover, a metric  $g$ with $R_g=\frac{4(n-1)}{n-2}\sigma_1(M),$ $H_g=0$ and $\mbox{Vol}(M,g)=1$ is Einstein with totally geodesic boundary. 
\end{theoremL}

\begin{theoremL}\label{prescribing thm2}
Let $M^n$ be a compact manifold with non empty boundary $\Sigma$, dimension $n\geq 3$, and $\sigma_0(M)\leq 0.$ 
\begin{enumerate}[(a)]
\item If $\sigma_0(M)> Y_0(M,C)$ for any conformal class $C$, then   $f\in C^\infty(\Sigma)$  is the mean curvature of some scalar flat metric $g$ on $M$  and $\mbox{Area}(\Sigma,g)=1$  if and  only if  
\begin{equation}\label{eq017}
    \min f<\frac{2}{n-2}\sigma_0(M).
\end{equation}
 
 \item If $\sigma_0(M)= Y_0(M,C)$ for some conformal class $C$, then   $f\in C^\infty(\Sigma)$  is the mean curvature of some metric $g$ on $M$ with $\mbox{Area}(\Sigma,g)=1$  if and  only if either 
    $f$ is constant equal to $\frac{2}{n-2}\sigma_0(M)$ or \eqref{eq017} holds.
\end{enumerate}
Moreover, a metric  $g$ with $R_g=0,$ $H_g=\frac{2}{n-2}\sigma_0(M)$ and $\mbox{Area}(\Sigma,g)=1$ is Ricci flat and has umbilical boundary. 
\end{theoremL}

For manifolds with  positive Yamabe invariant, we have the following.

\begin{theoremL}\label{positive}
Let $M^n$ be a compact manifold with non empty boundary, dimension $n\geq 3$ and $\sigma_1(M)>0$. Then, any function $f\in C^{\infty}(M)$ which is either nonconstant or constant less than $\sigma_1(M)$ is the scalar curvature of some metric with minimal boundary and unit volume.
\end{theoremL}

  It would be interesting to know if our result can be extended to the remaining cases, namely, for any positive constant
functions and for the case that $\sigma_0(M)>0$, providing a full analogy between the result of \cite{K} and this  boundary setting.

 Now if the conformal class is fixed, we obtain the following theorem that can be compared to Theorem 4 of \cite{K}.
 
\begin{theoremL}\label{finalthm}
Let $M$ be a compact manifold with boundary and dimension $n\geq 3$. Suppose that $Y_1(M,C)>0$  for some conformal class $C$. Then, given $\varepsilon>0$ and an integer  $k\geq 0$, there is a metric $g\in C$ with minimal boundary,
$\mbox{Vol}(M, g)= 1$ and satisfying 
$$|R_g-(Y_1(M,C)^{n/2}+ k(2\sigma_1(\mathbb S_+^n))^{n/2})^{2/n}|<\varepsilon.$$
\end{theoremL}

We observe that a metric which is obtained as a critical point of the functional \eqref{functional} satisfies the condition $(\mbox{min} R_g)\mbox{Vol}(M, g)^{2/n}\leq\sigma_1(\mathbb S^n_+)$. Thus, for $k>0$ and $\varepsilon>0$ sufficiently small, the metric given by Theorem \ref{finalthm} cannot  be obtained using variational method. We can ask if a similar result is true for manifold with $Y_0(M,C)>0$, i.e., if in this case there exists a scalar flat metric $g\in C$ with mean curvature of the boundary arbitrarily large.

\subsection*{Organization} The content of this paper is organized as follows. In Section \ref{preli},  we obtain equation \eqref{firststatic} through a variational characterization. Also we give some properties and examples of such metrics.
In Section \ref{constk}, we study the manifold of Riemannian metrics with constant curvature. In Section \ref{vari}, we construct explicit deformations on  the space of metrics defined in the previous section in order to prove Theorem \ref{teo1}. Finally, in Section \ref{prescsection}, we combine well known facts of the Yamabe invariant  with the purpose of solve the  modified Kazdan-Warner-Kobayashi problem for manifolds with boundary in several cases. 
Some technical details concerning a continuity property of the Yamabe invariant will be proved in one Appendix (Section \ref{appe}).

\noindent
{\bf Acknowledgements}.The authors would like to thank S. Almaraz for many useful comments on an earlier version of this paper and   E. Ribeiro Jr for several enlightening discussion about V-static metrics.

\section{Properties and computations}\label{preli}

In this section we  obtain equation \eqref{firststatic}, giving an analogous concept of V-static equation given in \cite{MT}, beside interpreting variationally and presenting examples of such  manifolds.

Now we fix  the basic notation that will be used throughout the paper.   
Let $(M,g)$ be a compact connected  Riemannian manifold with non empty boundary $\Sigma$. Given an infinitesimal variation $h,$ a well known calculation (cf. \cite{A}) shows that the variation of the scalar curvature $R_g$  and of the mean curvature $H_g$ in the direction of $ h$ is given by
 $$
\delta R_g h=\frac{d}{dt}\Big|_{t=0} R(g(t))=-\Delta_g(\tr_gh)+\mbox{div}_g \mbox{div}_g h - \langle h, \Ric_g\rangle
 $$
 and 
 $$\delta H_g h=\frac{d}{dt}\Big|_{t=0} H(g(t))=\frac{1}{2}\left( [d ( \tr_ g h ) - \mbox{div}_g h ] (\nu) - \Div_{g|_{T\partial M}} X - \langle \Pi_{g}, h \rangle\right),
 $$ where 
 $ \nu $ 
 is the  outward unit vector normal to the boundary $\Sigma$,
$X$ is the vector field dual to the one-form  $\omega(\cdot)=h(\cdot,\nu), $ $\tr_gh=g^{ij}h_{ij}$ is the trace of $h$  and  our convention for the laplacian is $\Delta_g f = \text{tr}_g(\mbox{Hess}_g f)$.

Consider the operator 
$$
\Psi(g)=(R_g, 2 H_{g}, -2\kappa\mbox{Vol}(g),-2\tau\mbox{Area}(g)),
$$
where $\kappa,\tau \in\mathbb{R}$. Its linearization  will be denoted by $\mathcal{S}_g(h).$  
A straightforward  calculation using the Green Formula (see  \cite[Section 2]{CV}) shows that 
\begin{equation}\label{green}
\langle\delta R_gh,V\rangle_{L^2(M)}-\langle A^*V,h\rangle_{L^2(M)}=
\langle B^*V,h\rangle_{L^2(\Sigma)}-\langle 2\delta H_g h,V\rangle_{L^2(\Sigma)},
\end{equation}
 where $A_g^*V  =  -(\Delta_g V)g+\mbox{Hess}_gV-V\Ric_g$
in $M$ and $B_g^*V  = \displaystyle \frac{\partial  V}{\partial  \nu}g-V\Pi_{g}$ on $\Sigma$. Since the  linearization of the area and of the volume functional are given by $\delta\mbox{Area}_g\cdot h=\displaystyle\frac{1}{2}\int_{\Sigma}\mbox{tr}_g (h|_{\Sigma}) da$  and    $\delta\mbox{Vol}_g\cdot h=\displaystyle\frac{1}{2}\int_{ M}\mbox{tr}_gh dv$, respectively, then the formal $L^2$-adjoint of $\mathcal{S}_g$, denoted by  $\mathcal{S}_g^*,$ is given by
$$
\mathcal{S}_g^*(V,\kappa,\tau)=(A_g^*V-\kappa g, B^*_g V- \tau g).
$$

Thus a triple $(V,\kappa,\tau)$ in the kernel of $\mathcal{S}_g^*$ satisfies the PDE \eqref{firststatic}  which in fact is a second order overdetermined elliptic equation  with  oblique boundary value condition. Taking the trace of \eqref{firststatic} we have that
 \begin{equation}\label{trace}
\left\{
  \begin{array}{rcll}
\displaystyle \Delta V+\frac{R_g}{n-1}V  & = &  \displaystyle -\frac{\kappa n}{n-1} &\mbox{ in } M\\
\vspace{-0,2cm}\\
\displaystyle\frac{\partial V}{\partial \nu}-\frac{H_{g}}{n-1}V & = & \tau & \mbox{ on } \Sigma.
\end{array}
  \right.
 \end{equation}

\begin{remark}
 A triple $(M,g,V)$ satisfying \eqref{firststatic} is also called in the literature of \textit{singular space}, provided $\mbox{Ker}~\mathcal{S}_g^*\neq \{0\}$. If we assume  constraints on the volume and on the area, we can  show that non-singular spaces are actually linearized stable, in the sense that $g\mapsto (R_g,2H_g)$ is a submersion. In others words we can locally prescribe the curvature of such  manifolds  with constant volume or constant area of the boundary. This will follow from  a slight modification of the proof of Proposition 3.1 and  3.3 of \cite{CV}, see also Section \ref{prescsection} for a more general approach. 
 
 Recently P. T. Ho and Y.-C. Huang \cite{HH} have studied the problem of prescribing the scalar curvature in a compact manifold $M$ and the mean curvature on the boundary $\partial M$ simultaneously, provided the manifold is not a singular space.
\end{remark}

With the appropriate analogous of V-static manifolds with boundary at hand, we follows closely the ideas of Theorem 3.2 of \cite{MT} in order to prove some properties for metrics satisfying \eqref{firststatic}.

\begin{proposition}\label{properties} Let $(M,g)$ be a
connected Riemannian compact manifold with non empty boundary $\Sigma$. Suppose there is a  non-identically zero function $V$ in the interior and on $\Sigma$  satisfying \eqref{firststatic} for some $\kappa,\tau\in \mathbb R.$ Then the following assumptions are true.
\begin{enumerate}[(a)]
\item The scalar curvature $R_g$  is constant.
\item\label{item001} The mean curvature $H_g$ is constant and $\Pi_g =  \dfrac{H_g}{n-1}g$.
\item\label{item013}  At each point of $\Sigma$ it holds
          \begin{equation*}\label{Gauss}
           R_{\Sigma}-\frac{n-2}{n-1}H^2_g=R_g-2\Ric_g(\nu,\nu).
         \end{equation*}
         
\item\label{item014} Consider the following functional  on the space of Riemannian metrics
\begin{align*}
\mathcal{F}(g)=& \int_{ M} R_g V\;dv+ 2\int_{\Sigma} H_g V\; da-2\kappa  \mbox{Vol}(g)-2\tau\mbox{Area}(\Sigma,g),
\end{align*}
where $V$ is a given smooth nontrivial function  on $M.$ Then $g$ is a critical point of $
\mathcal{F}.$         
\end{enumerate}
\end{proposition}

\begin{proof} 
 The proof that   $R_g$ is constant is standard as in A. E. Fischer and J. E. Marsden \cite{FM}. We observe that the divergence of the first equation in \eqref{firststatic} gives   $V d R_g =0.$ If $V$ is never zero then $R_g$ is constant. So, suppose $V$ is zero in some point $p\in M$. Restricting $V$ to unit geodesics $\gamma$ with $\gamma(0)=p$ we obtain a second order ODE with initial data $V(p)$ and $dV(\gamma
'(0))$. In case $\kappa=0$, if $V(p)=dV(\gamma'(0))=0$, then $V\equiv 0$ along $\gamma$. This implies that the zero set of $V$ is a submanifold of $M$ with codimension one. It follows that $dR_g=0$, so $R_g$ is constant. In case $\kappa\not=0$, a solution of the inhomogeneous ODE cannot be identically zero in a non-empty open set. Again we obtain that $R_g$ is constant.

 In order to prove \eqref{item001}, we modify slightly the arguments in Proposition 3.1 of \cite{CV}. Since  $V$ is not identically zero on $\Sigma$, it follows from \eqref{trace} that $\Pi_{g}=\frac{H_{g}}{n-1}g.$ 
Suppose that $\{e_i\}_{i=1}^{n-1}$ span $T\Sigma$ locally and consider $e_n=\nu$. Then we get by Codazzi equation that 
\begin{eqnarray}\label{codazzi}
\nabla_i^{\Sigma}H_{g}&=&\nabla_i^{\Sigma}\Pi^j_j=\nabla_j^{\Sigma}\Pi^j_i+R^j_{ji\nu}=\nabla_j^{\Sigma}\Pi^j_i+R_{i\nu}\nonumber,
\end{eqnarray}
where $R_{ijkl}$ is the curvature tensor of $(M,g).$ Thus,
 \begin{equation}\label{eq001}
     \nabla_i^{\Sigma}H_{g}=\nabla_j^{\Sigma}\Big(\frac{H_{g}}{n-1}g_i^j\Big) +R_{i\nu}=\frac{1}{n-1}\nabla_i^{\Sigma_0}H_{g} +R_{i\nu}.
 \end{equation}
On the other hand,  we have that 
\begin{eqnarray*}
0&=& \nabla_i\Big(\frac{\partial V}{\partial \nu}g^i_j-V\Pi^i_j-\tau g^i_j\Big)\\
&=&V_{i\nu} -\frac{V}{n-1}\nabla_i^{\Sigma}H_{g}.
\end{eqnarray*}
By the first equation in \eqref{firststatic}, we get $\Hess_g V(e_i,\nu)=V R_{i\nu}$. Thus, by \eqref{eq001} we have that $\nabla^{\Sigma} H_{g}=0,$ so $H_{g}$ is constant.

Finally,   item \eqref{item013} follows by the Gauss equation and  item \eqref{item014} is immediate from \eqref{green}.
\end{proof}
 \begin{remark}
 There are several motivations to consider the \textit{weighted  curvature functional} $\mathcal{F}$ given in item \eqref{item014} of Proposition \ref{properties}, which was first introduced by  A. E. Fischer and J. E. Mardsen in \cite{FM}. Indeed, this  functional  plays a fundamental role in the theory of deformation and rigidity of static and V-static manifolds as we can observe in the works by S. Brendle et al \cite{BMN}, S. Brendle and F. C. Marques \cite{BM}, G. Cox, P. Miao and L.-F. Tam \cite{CMT}  and  J. Qing and W. Yuan \cite{QY}.
 \end{remark}

The next proposition shows the importance of the 
spectrum of the Steklov and Neumann to the characterization of metrics satifying equation \eqref{firststatic}.

\begin{proposition} Let $(M,g)$ be a connected Riemannian compact manifold with boundary $\Sigma$.
\begin{enumerate}[(a)]
\item\label{item015} Assume that $g$ is a scalar flat metric with nonzero constant mean curvature on $\Sigma$.  
If $\Sigma$ is not umbilical and $\frac{H_g}{n-1}$ is not in the spectrum of the  Steklov eigenvalue problem, then there is no function $V$ satisfying equation \eqref{firststatic} in $M$ with $\kappa=0$ and $\tau=1$.
\item Assume that $g$  is a metric with nonzero constant scalar curvature  with zero mean curvature on $\Sigma$.  
If $g$ is not Einstein and does not admit $\frac{R_g}{n-1}$ in the spectrum of the  Neumann eigenvalue problem, then there is no function $V$ satisfying equation  \eqref{firststatic}  with $\kappa=1$ and $\tau=0$.
\end{enumerate}
\end{proposition}
\begin{proof}
Let $V\in C^{\infty}(M)$ be a solution of \eqref{firststatic}  with $\kappa=0$ and $\tau=1$. By Proposition \ref{properties}, the scalar curvature and the mean curvature are constants. Using the trace \eqref{trace} we obtain that 
 \begin{equation*}
\left\{
  \begin{array}{ll}
\displaystyle\Delta_g \left(V+\frac{n-1}{H_g}\right)= 0 &\mbox{ in } M\\
\displaystyle\frac{\partial }{\partial \nu}\left(V+\frac{n-1}{H_g}\right)=\frac{H_{g}}{n-1}\left(V+\frac{n-1}{H_g}\right)  & \mbox{ on } \Sigma.
\end{array}
  \right.
 \end{equation*}
Thus  $V+\frac{n-1}{H_g}\equiv 0$, which implies by the second equation in \eqref{firststatic} that $\Sigma$ is umbilical, contradicting the hypothesis. This completes part \eqref{item015}.

Now let $V\in C^{\infty}(M)$ be a solution of \eqref{firststatic}  with $k=0$ and $\tau=1$. Again, using the trace \eqref{trace} we obtain that
 \begin{equation*}
\left\{
  \begin{array}{ll}
\displaystyle\Delta_g \left(V+\frac{n}{R_g}\right)= -\frac{R_{g}}{n-1}\left(V+\frac{n}{R_g}\right)  & \mbox{ in } M\\
\displaystyle\frac{\partial }{\partial \nu}\left(V+\frac{n}{R_g}\right)= 0& \mbox{ on }\Sigma.
\end{array}
  \right.
 \end{equation*}
 Arguing as before we obtain the desired result for item (b).
\end{proof}

 \subsection{Some examples}

Next, we present some examples of functions  satisfying equation \eqref{firststatic}, such functions  will be called of \textit{potential functions}. 
\begin{example}
Let $\Omega$ be a geodesic ball in $\mathbb{S}^n$ in the Euclidean space with center at $N=(0,\ldots,0, 1)$ and geodesic radius $R\in(0,\pi)$. If $r$ is the geodesic distance to $N$, one can check that the following function 
$$V=a\cos r-\frac{\kappa}{n-1},$$
satisfies \eqref{firststatic} for
$\tau=a\frac{\cos 2R}{\sin R}-\frac{\kappa}{n-1}\frac{\cos R}{\sin R}$, where $a\in\mathbb R$.

\end{example} 

\begin{example}
Let $\mathbb B$ be the unit Euclidean ball in $\mathbb{R}^n.$ The function $V$ given by
$$
V=-\frac{\kappa}{2(n-1)}|x|^2+\langle b,x\rangle-\tau,
$$
where $\kappa,\tau\in\mathbb R$ and $b\in\mathbb R^n$, is a potential for equation \eqref{firststatic}. 

\end{example}

 \begin{example}\label{ex002}
 Let $\mathbb{L}^{n+1}=(\mathbb R^{n+1},ds^2)$ be the Minkowski space with the metric $ds^2=dx_1^2+\ldots+dx_n^2-dt^2.$ Consider  $$\mathbb{H}^{n}=\left\{\left(x_{1}, \ldots, x_{n}, t\right) \in \mathbb{R}^{n+1} ; \sum_{i=1}^{n} x_{i}^{2}-t^{2}=-1, t \geq 1\right\}$$ 
embedded in $\mathbb{L}^{n+1}$ with the induced metric $g$. Fix $p=(0, \ldots, 0,1) \in \mathbb{H}^{n},$ the geodesic ball $\Omega^{n} \subset \mathbb{H}^{n}$ with center $p$ and radius $R_{0}$. If $r$ is the geodesic distance to $p$, one can check that the following function 
$$V=a\cosh r+\frac{\kappa}{n-1},$$
satisfies \eqref{firststatic} for $\tau=a\frac{\cosh 2R}{\sinh R}+\frac{\kappa}{n-1}\frac{\cosh R}{\sinh R}$, where $a\in\mathbb R$.
 
 \end{example}

 \begin{example}
Consider the model of the hyperbolic space $\mathbb H^n$ given by \linebreak Example \ref{ex002}.
Let $\left(\mathbb{H}_{+}^{n}, b, \partial \mathbb{H}_{+}^{n}\right),$ where $\mathbb{H}_{+}^{n}=\{x \in\mathbb {H}^{n} ; x_{n} \geq 0\}$ be the hyperbolic half-space endowed with the induced metric
$
b=dr^{2}/(1+r^{2})+r^{2} h_{0},
$
where $h_{0}$ is the canonical metric on the unit hemisphere $\mathbb{S}_{+}^{n-1},$ and
$
r=\sqrt{x_{1}^{2}+\cdots+x_{n}^{2}}.
$
The space of static potentials on $\mathbb{H}^{n}$, denoted by $\mathcal{N}_{b}$, is spanned by $V_{(0)}, V_{(1)}, \cdots, V_{(n)},$ where $V_{(i)}=x_{i}\mid_{\mathbb{H}^{n}}.$ Since for each $i \neq n$
$$
\frac{\partial V_{(i)}}{\partial \nu}=0,
$$
where $\nu$ is the outward unit normal to $\partial \mathbb{H}_{+}^{n},$ we can define the  space of
static potentials of $\mathbb{H}_{+}^{n}$ as
$$
\mathcal{N}_{b}^{+}=\left\{V \in \mathcal{N}_{b} ; \frac{\partial V}{\partial \nu}=0\right\},
$$
which is spanned by $V_{(0)}, V_{(1)}, \cdots, V_{(n-1)}.$ We highlight  the importance of $\mathcal{N}_{b}^{+}$  to define the mass of an \textit{Asymptotically Hyperbolic} manifold with non-compact boundary as we can see in \cite{AL}. 
We are grateful to S. Almaraz for suggesting this motivation coming from
  general relativity theory .
 \end{example}

 \begin{example}
 Every Ricci-flat metric on $M$ with totally geodesic boundary satisfies \eqref{firststatic} for $\kappa=\tau=0$. Moreover,  we can see that the potential function is generated by 1. Since we can scale the volume and area of the boun-dary such that the metric remains Ricci-flat with totally geodesic boundary,  such  metric cannot be critical for the volume and the area of the boundary    functional.
 \end{example}
 
For more examples with $\kappa=\tau=0,$ see Section 4 of \cite{HH}.

\begin{remark}
We left to the reader to compare the difference between these examples with those in  \cite{barros2015,BDR,MT,MT2}, where the V-static case is treated. 
\end{remark}
 
\section{Manifolds of metrics of prescribed curvature }\label{constk}

Let $M^n$ be an $n$-dimensional  compact connected  Riemannian manifold with boundary $\Sigma$. 
Let $S^{k,2}_2=W^{k,2}(\mbox{Sym}^2(T^*M))$ be the section of class $W^{k,2}$ of  symmetric $(0, 2)$-tensors $h$. For $k>\frac{n}{2}+2,$ consider the operator 
$$
\Phi(\cdot):=(R(\cdot), 2H(\cdot)):\mathcal{M}^{k,2}\to W^{k-2,2}(M)\oplus W^{k-\frac{3}{2},2}(\Sigma),
$$
 where  $\mathcal{M}^{k,2}$ denotes the open subset of $S^{k,2}_2$ of metrics on $M$. Since $R_g$ and  $H_{g}$  involve derivatives of $g$ up to second order, the local expression of the scalar curvature and the mean curvature, for $k>\frac{n}{2}+2,$ gives that the above Sobolev spaces  are a Banach algebra under pointwise multiplication \cite{MS}, which implies that  $\Phi$  is a $C^{\infty}$ map.

As it is known,  given a compact manifold, there exist many scalar flat metrics with constant mean curvature on the boundary and constant scalar curvature with minimal boundary (see for instance \cite{E2,E1}). Then, for a smooth function $\rho: \Sigma\to\mathbb{R}$ or $\rho:  M\to\mathbb{R}$ we can  set 
$$\mathcal{M}_{0,\rho}=\{g\in \mathcal{M}^{k,2};\quad R_g=0\quad\mbox{and}\quad H_{g}=\rho\}$$
or
$$\mathcal{M}_{\rho,0}=\{g\in \mathcal{M}^{k,2};\quad R_g=\rho\quad\mbox{and}\quad H_{g}=0\}.,$$
respectively. The spaces  $\mathcal{M}_{0,\rho}$  and $\mathcal{M}_{\rho,0}$ are called set of the space of metrics with prescribed curvature.

We will say that $\sigma\in\mathbb R$ is an Steklov eigenvalue of the boundary operator $\partial/\partial\nu-c/(n-1)$, where $c$ is a constant, with nonzero eigenfunction $u$ if 
  \begin{equation} \label{stek0}
\left\{
  \begin{array}{rll}
\Delta u & = 0 & \mbox{ in } M\\
\dfrac{\partial u}{\partial \nu}-\dfrac{c}{n-1}u & =\sigma u & \mbox{ on }\Sigma.
\end{array}
  \right.
 \end{equation}

We prove the following result.

\begin{proposition}\label{subman} Let $(M,g_0)$ be a  compact manifold with non empty boun-dary $\Sigma$. Assume that $g_0$ is a scalar flat metric with constant mean curvature $c$ on the boundary. 
If the first Steklov eigenvalue of $\partial/\partial\nu-c/(n-1)$ is positive, then $\mathcal M_{0,c}$
%$$\mathcal{M}_{c}=\{g\in \mathcal{M}^{k,2};\quad R_g=0\quad\mbox{and}\quad H_{g_{|_{T(\partial M)}}}=c\},$$ 
 is a smooth submanifold
 of $\mathcal{M}^{k,2}$ for $g$ near $g_0.$
\end{proposition}
\begin{proof}
 Since the first eigenvalue of  \eqref{stek0} is positive, the Fredholm alternative (cf. \cite{GT}) implies the existence of a unique solution $u$ of  the following problem
\begin{equation*}\label{dirich0} 
\left\{
  \begin{array}{rll}
\Delta u & = 0 &\mbox{ in }M\\
\dfrac{\partial u}{\partial \nu}-\dfrac{c}{n-1}u & = \dfrac{f}{n-1} &\mbox{ on } \Sigma,
\end{array}
  \right.
 \end{equation*}
 where $f\in W^{k-\frac{3}{2},2}(\Sigma).$   Set $\Phi(g)=(R_g, 2H_g). $ If we take   $h = ug_0,$ we get that $D\Phi_{g_0}(h) = (0,f),$ i.e., $D\Phi_{g_0}$ is surjective under the second coordinate.   Since the kernel of $D\Phi_{g_0}(h)$ splits\footnote{The split of kernel of $D\Phi_{g_0}$  means  that $\mathcal{M}^{k,2}= \mbox{ker}D\Phi_{g_0}(h)\oplus Y,$ for some closed  subspace $Y\subset\mathcal{M}^{k,2}$ with $Y\cap \mbox{ker}D\Phi_{g_0}(h)=\emptyset.$} The result follows by the Implicit Function Theorem.
\end{proof}

Analogously,
we  consider the  Schr\"odinger operator $\Delta+c/(n-1)$ with  Neumann boundary condition, and for that one have the following result:

\begin{proposition}\label{subman2} Let $(M,g_0)$ be a  compact manifold with boundary.\linebreak
  Assume that $g_0$ is  a metric with constant scalar  curvature $c$ and minimal boundary. 
If the first eigenvalue of $(n-1)\Delta_{g_0}+c$   is positive with  Neumann  boundary, then $\mathcal M_{c,0}$ is a smooth submanifold
 of $\mathcal{M}^{k,2}$ for $g$ near $g_0.$
\end{proposition}

\begin{remark}
Along the same lines, it is possible to show under the same conditions as above that the sets defined in the following are submanifolds of $\mathcal{M}^{2,p}\cap\{g; \mbox{ Vol}(M,g)=1\}$ and  $\mathcal{M}^{2,p}\cap\{g; \mbox{ Area}(\partial M,g)=1\}:$ 
$$\Xi_{0,c}=\{g\in \mathcal{M}^{2,p}\cap\{g; \mbox{ Vol}(M,g)=1\};\; R_g=0\quad\mbox{and}\quad H_{g}=c\}$$
and   
$$\Xi_{c,0}=\{g\in \mathcal{M}^{2,p}\cap\{g; \mbox{ Area}(\partial M,g)=1\};\; R_g=c\quad\mbox{and}\quad H_{g}=0\}.$$
Indeed, it is sufficient to consider the functional $$\Psi(g)=\left(R_g-\int_MR_gdv,2H_g-2\int_{\partial M}H_g da\right).$$ Note that $\Psi(g)=0$ if and only if $R_g$ and $H_g$ are constants.

\end{remark}

\section{Variational point of view and proof of Theorem \ref{teo1}}\label{vari}

In this section, we   consider the  problem of finding stationary points for the volume functional and  area of the boundary  functional on the space of  prescribed metrics considered in  Section \ref{constk}.

We start this sections with the following lemma, which is motivated by the proof of Theorem 1 in \cite{Fischer-Colbrie-Schoen}.

\begin{lemma}\label{help}
 Let $g_0\in \mathcal M_{0,c}$ be a metric such that the first eigenvalue of \eqref{stek0}  is positive. Then there exist a positive function $u$ on $M$ and constants  $\delta_0,\delta>0$ such that
 \begin{equation}\label{eq006}
     \left\{ \begin{array}{rcl}
\Delta u+\delta_0u & = & 0\mbox{ in }M\\
\displaystyle \frac{\partial u}{\partial \nu}-\frac{c+\delta}{n-1} u & = & 0 \mbox{ on } \Sigma.
\end{array}
  \right.
 \end{equation}
\end{lemma}

\begin{proof}
Fix a point $x_0\in M.$ Let $\Omega_R=\{(M\setminus\Sigma)\cap B_R(x_0)\}.$
Let $\partial \Omega_R^D=\{\partial \Omega_R\setminus\Sigma\}$ and $\partial \Omega_R^N=\{\partial \Omega_R\setminus\Omega_R^D\}.$ 
In this context, there is an associated mixed boundary value problem and since clearly a  monotonicity property holds for the local eigenvalue, $\sigma_1(\Omega_R)\geq\sigma_1(M)>0$, 
the Fredholm alternative (see Lieberman \cite{Lie}) gives a solution of function $v\in C^2(\Omega_R\cup \partial \Omega_R^N)\cap C^0(\overline{\Omega_R})$ such that 

$$
\left\{
  \begin{array}{rcll}
\Delta v +\delta_0v& = & -\delta_0 & \mbox{in}\quad \Omega_R\\
\displaystyle \frac{\partial v}{\partial\nu}-\frac{c+\delta}{n-1}v & = & \displaystyle\frac{c+\delta}{n-1}  & \mbox{on}\quad \partial \Omega_R^N\\
v& = &0  & \mbox{on}\quad \partial \Omega_R^D
\end{array}
  \right.
$$
Moreover,  in the vertex $\partial \Omega_R^D\cap\Sigma$ we have that $\nu_{\Omega_R}\cdot\nu<0,$ then $v\in C^1(\overline{\Omega_R})$ by Lieberman \cite{Lie2}
for some  constants $\delta_0,\delta >0$.  Using elliptic Schauder estimates there exists a smooth function $v$  from the vertex $\partial \Omega_R^D \cap \Sigma$. Then, consider $u=v+1 $ which is a solution of 
$$
\left\{
  \begin{array}{rcl}
\Delta u+\delta_0u&=& 0\quad\mbox{in}\quad \Omega_R\\
\displaystyle\frac{\partial u}{\partial \nu}-\frac{c+\delta}{n-1}u&=&0  \quad\mbox{on}\quad \partial \Omega_R^N\\
u&=&1  \quad\mbox{on}\quad\partial  \Omega_R^D.
\end{array}
  \right.
$$

 We claim that  $u$ is a positive function. Suppose that $u \geq 0$, then by the strong maximum principle, we obtain that $u>0.$
Suppose that $D\subset \{x \in \Omega_R; u(x)<0\}$ is non empty. Since $D$ is a bounded domain,  the first eigenvalue  $\sigma_1(D)$ is positive and, thus, $\Delta u+\delta_0u=0$ in $D$ and $u=0$ on $\partial D.$ Hence $v=0$ on $D$. This contradicts the unique continuation property and therefore $u>0.$

\end{proof}

Now, we state the following key ingredient in the proof of Theorem \ref{teo1}, whose proof is inspired by Proposition 2.1. of \cite{MT}. It proceeds by a careful analysis using the  sub- and super-solutions methods. 

\begin{proposition} \label{target}
Let $(M,g_0)$ be a compact Riemannian manifold with non-empty boundary  $\Sigma$.
  Assume that $g_0\in \mathcal{M}_{0,c}$  and  the first  eigenvalue of   \eqref{stek0} with respect to $g_0$ is positive. Let $g(t)=g_0+th$ be a smooth one-parameter family of Riemannian metrics for $|t|$ small enough and $h$ a smooth symmetric $(0,2)$-tensor on $M$. Then there  exist constants $t_0>0$ and $\varepsilon>0$ such that for $|t|<t_0$ there exists  a unique 
smooth positive function $\Phi(t)$ on $M$ such that $|\Phi(t) - 1|\leq \varepsilon$  and
\begin{equation}\label{neuscal}
\left\{
  \begin{array}{rcl}
\displaystyle \Delta_{g(t)} \Phi(t)-\beta R_{g(t)}\Phi(t)&=& 0\quad\mbox{in}\quad M\\
\displaystyle\frac{\partial \Phi(t)}{\partial \nu}+2\beta H_{g(t)}\Phi(t) -\displaystyle 2\beta c\Phi(t)^{\frac{n}{n-2}}& = & 0 \quad\mbox{on}\quad\Sigma,
\end{array}
  \right.
 \end{equation}
 where $\beta=\frac{n-2}{4(n-1)}.$ 
\end{proposition}
\begin{proof}
For each $t$ we define the following boundary value operator 
\begin{equation*}
(t,v)\mapsto\left\{
  \begin{array}{rcll}
\mathcal Lv&:=& \displaystyle \Delta_{g(t)} v-\beta R_{g(t)}v & \mbox{in}\quad M,\\
\mathcal Bv&:=&\displaystyle \frac{\partial v}{\partial \nu}+2\beta H_{g(t)}v  -2\beta cv^\frac{n}{n-2}  & \mbox{on}\quad\Sigma.
\end{array}
  \right.
 \end{equation*} 
Since  the first eigenvalue of \eqref{stek0}  is positive, we can take a positive function  $u$ as in Lemma \ref{help}. Then for such a function
$$\begin{array}{rcl}
     \mathcal B(1 +tu) & = & \displaystyle t\langle \nabla_{g(t)}u,\nu\rangle+2\beta H_{g(t)}(1 + tu) -2\beta c(1 + tu)^\frac{n}{n-2}\\
     \\
     & = &\displaystyle\delta tu+ t\langle (\nabla_{g(t)}-\nabla_{g(0)})u,\nu\rangle +2\beta (H_{g(t)}-c)(1 + tu)\\
     \\
     & & \displaystyle  +\frac{c}{n-1} tu+2\beta c(1+tu) -2\beta c(1 + tu)^\frac{n}{n-2}.
\end{array}$$
Also we have
$$\begin{array}{rcl}
     \mathcal B(1 -tu) & = & \displaystyle -t\langle \nabla_{g(t)}u,\nu\rangle+2\beta H_{g(t)}(1 - tu) -2\beta c(1 - tu)^\frac{n}{n-2}\\
     \\
     & = &\displaystyle-\delta tu- t\langle (\nabla_{g(t)}-\nabla_{g(0)})u,\nu\rangle +2\beta (H_{g(t)}-c)(1 - tu)\\
     \\
     & & \displaystyle  -\frac{c}{n-1} tu+2\beta c(1-tu) -2\beta c(1 - tu)^\frac{n}{n-2}.
\end{array}$$
Thus, for $t>0$ small enough we get
$$\mathcal B(1 +tu)\geq \delta tu-C_1t-C_2t^2$$
and
$$\mathcal B(1 +tu)\leq -\delta tu+C_1t+C_2t^2,$$
where $C_1=C_1(g_0,h)$ and $C_2=C_2(g_0,h,u)$ are positive constants. Since $b=\min_Mu>0$ and $u$ solves \eqref{eq006}, then by rescaling $u$, we may assume that $\delta b>2C_1$. Thus, for $t>0$ small enough
\begin{equation*}\label{eq007}
    \mathcal B(1 - tu)\leq0\leq \mathcal B(1 +tu)\quad \mbox{on}\quad \Sigma.
\end{equation*}

 It remains to verify the behaviour in the interior. Note that
 $$\begin{array}{rcl}
\mathcal L(1 +tu) & = & \displaystyle t\Delta_{g_{0}} u+t\left(\Delta_{g(t)}-\Delta_{g_{0}}\right) u-\beta R_{g(t)}(1+tu)
 \end{array}$$
 and
  $$\begin{array}{rcl}
\mathcal L(1 -tu) & = & \displaystyle -t\Delta_{g_{0}} u-t\left(\Delta_{g(t)}-\Delta_{g_{0}}\right) u-\beta R_{g(t)}(1-tu).
 \end{array}$$
 
As before, using that $u$ satisfies \eqref{eq006}, by rescaling, we get that for $t>0$ small enough
 \begin{equation*}\label{eq008}
     \mathcal L(1+tu)\leq 0\leq \mathcal L(1-tu) \quad\mbox{in}\quad M.
 \end{equation*}

Then  by the method of sub- and super-solutions  (cf. \cite[Theorem 2.3.1]{Sa} for example) there exists a solution of the boundary value problem  \eqref{neuscal} with $1 - tu\leq\Phi(t)\leq1 + tu$ provided $t>0$. The proof for $t<0$ is similar.

In order to prove uniquenesss, assume that  $\Phi_1(t),\Phi_2(t)\in (1-\varepsilon,1+\varepsilon)$ are solutions of 
\eqref{neuscal}. Since the first eigenvalue of \eqref{stek0} is positive,   
there exists $\delta=\delta(g_0,h)>0$ such that 
\begin{align*}
\int_{\Sigma}\left(\Phi_{1}(t)-\Phi_{1}(t)\right) \left(\frac{\partial }{\partial \nu}(\Phi_{1}(t)-\Phi_{2}(t))-\frac{c}{n-1}(\Phi_{1}(t)-\Phi_{2}(t)) \right) da_{g(t)}\\
\geq \delta \int_{\Sigma}\left(\Phi_{1}(t)-\Phi_{2}(t)\right)^{2} da_{g(t)}.
\end{align*}
On the other hand, since 
\[
  \int_{\Sigma}(\mathcal B(\Phi_{1}(t))-\mathcal B(\Phi_{2}(t)))(\Phi_{1}(t)-\Phi_{2}(t))da_{g(t)}=0,
\]
we have 
\begin{align*}
\int_{\Sigma}\left(\Phi_{1}(t)-\Phi_{1}(t)\right)& \left(\frac{\partial }{\partial \nu}(\Phi_{1}(t)-\Phi_{2}(t))-\frac{c}{n-1}(\Phi_{1}(t)-\Phi_{2}(t)) \right) da_{g(t)}\\
&=2\beta \int_{\Sigma} H_{g(t)}(\Phi_{1}(t)-\Phi_{2}(t))^{2} da_{g(t)} \\
& -\frac{c}{n-1}\int_{\Sigma}(\Phi_{1}(t)-\Phi_{2}(t))^2 da_{g(t)}\\
&+2\beta c\int_{\Sigma}\left(\Phi_{1}(t)^{\frac{n}{n-2}}-\Phi_{2}(t)^{\frac{n}{n-2}}\right)\left(\Phi_{1}(t)-\Phi_{2}(t)\right)da_{g(t)} \\
& \leq\left(K_1|t|+K_2 \varepsilon\right) \int_{\Sigma}\left(\Phi_{1}(t)-\Phi_{2}(t)\right)^{2} da_{g(t)},
\end{align*}
where $K_1$ and $K_2$ are constants depending only on $g_0$ and
$h$.  
Hence we have 
$$
\delta \int_{\Sigma}\left(\Phi_{1}(t)-\Phi_{2}(t)\right)^{2} da_{g(t)}\leq \left(K_5|t|+K_6 \varepsilon\right) \int_{\Sigma}\left(\Phi_{1}(t)-\Phi_{2}(t)\right)^{2} da_{g(t)},
$$
which for $|t|$ and $\varepsilon>0$ sufficiently small enough we get that
 $\Phi_1=\Phi_2$.
\end{proof}

\begin{corollary}\label{derivative}
Under the same conditions of Proposition \ref{target}, $\hat\Phi=\frac{d}{dt}\big|_{t=0}\Phi$ exists and is the unique smooth function on $M$ satisfying 
\begin{equation}\label{deriv}
\left\{
  \begin{array}{rcl}
\displaystyle \Delta_{g_0} \hat\Phi-\frac{n-2}{4(n-1)} R'(0)&=& 0\quad\mbox{in}\quad M\\
\displaystyle\frac{\partial \hat\Phi}{\partial \nu}+\frac{n-2}{2(n-1)}H'(0)-\frac{c}{n-1}\hat\Phi&=&0 \quad\mbox{on}\quad\Sigma.
\end{array}
  \right.
 \end{equation}
\end{corollary}
\begin{proof}

Let $\Phi(t)$ be
a  solution of \eqref{neuscal} for $t\neq0$ given by Proposition \ref{target} and let $u$ be a function given by Lemma \ref{help}. For this solution  $|\Phi(t)-1|\leq |t|u$. 
If we set $w(t)=(\Phi(t)-1)/t$, this satisfies the following boundary equation:

\begin{equation}\label{def}
\left\{
  \begin{array}{rcll}
  \Delta_{g(t)} w(t)&=&\displaystyle-\frac{n-2}{4(n-1)}\left(\frac{R_{g(t)}}{t}+R_{g(t)}w(t)\right) &\text{\ in $M$}\\
  \\
\displaystyle     \frac{\partial w(t)}{\partial \nu}&=&\displaystyle \frac{n-2}{2(n-1)}\left(\frac{c - H_{g(t)}}{t}-H_{g(t)}w(t)\right.\\
\\
& &\displaystyle\left. +c\frac{ \Phi(t)^{\frac{n}{n-2}}-1 }t\right) & \text{\ on $\Sigma$}
     \end{array}
  \right.
 \end{equation}

We observe that there exists a constant, which does not depend on $t$ and $M$, so that bounds the  right side of \eqref{def}. Then by applying standard H\"older  and the  Schauder estimates for the oblique boundary  value problems  (see \cite[Theorem 8.29]{GT} and \cite[Lemma 6.29]{GT}), we can find a subsequence which converge to a solution $\hat \Phi$ of \eqref{def} when $t_j\to0$. Since $R_{g(0)}=0$ and the first  eigenvalue of   \eqref{stek0} with respect to $g_0$ is positive, we have  uniqueness by  Fredholm alternative, and hence  $\hat\Phi=\frac{d}{dt}\big|_{t=0}\Phi.$
\end{proof}

Alternatively we consider the following result.

\begin{proposition} \label{target2}
Let $(M,g_0)$ be a compact Riemannian manifold with non-empty boundary  $\Sigma$. Assume that $g_0\in \mathcal{M}_{c,0}$  and  the first  eigenvalue with Neumann  boundary of $\Delta_{g_0}+\frac{c}{n-1}$   is positive. Let $g(t)=g_0+th$ be a smooth one-parameter family of Riemannian metrics for $|t|$ small enough and $h$ a smooth symmetric $(0,2)$-tensor on $M$. Then there  exist constants $t_0>0$ and $\varepsilon>0$ such that, for $|t|<t_0,$ there exists  a unique 
smooth positive function $\Phi(t)$ on $M$ such that $|\Phi(t) - 1|\leq \varepsilon$  and
\begin{equation}\label{dir}
\left\{
  \begin{array}{rcl}
\displaystyle \frac{4(n-1)}{n-2}\Delta_{g(t)} \Phi(t)- R_{g(t)}\Phi(t)+ c\Phi(t)^{\frac{n+2}{n-2}}&=& 0\quad\mbox{in}\quad M\\
\displaystyle\frac{\partial \Phi(t)}{\partial \nu}+\frac{n-2}{2(n-1)} H_{g(t)}\Phi(t) & = & \displaystyle 0 \quad\mbox{on}\quad\Sigma,
\end{array}
  \right.
 \end{equation}
Moreover, $\hat\Phi=\frac{d}{dt}\big|_{t=0}\Phi$ exists and is the unique smooth function $\hat{\Phi}$ on $M$ satisfying 
\begin{equation}\label{deriv2}
\left\{
  \begin{array}{rcl}
\displaystyle \Delta_{g_0} \hat\Phi-\frac{n-2}{4(n-1)}R'(0)+\frac{c}{n-1}\hat\Phi&=& 0\quad\mbox{in}\quad M\\
\displaystyle\frac{\partial \hat\Phi}{\partial \nu}+\frac{n-2}{2(n-1)}H'(0)&=&0 \quad\mbox{on}\quad\Sigma.
\end{array}
  \right.
 \end{equation}

\end{proposition}

 Due to the similarity of the proof of Proposition \ref{target},  we merely sketch the proof of Proposition \ref{target2} whose details we left to the reader. First,  the eigenvalue condition implies the existence of a positive function $u$ on $M$ and constants $\delta_0,\delta>0$ such that
 $$
\left\{
  \begin{array}{rccl}
\displaystyle \Delta u+\frac{c}{n-1}u +\delta u& = & \displaystyle 0 & \mbox{in } M\\
\displaystyle \frac{\partial u}{\partial \nu}-\delta_0 u & = & 0  & \mbox{on }  \Sigma
\end{array}
  \right.
$$

Analogously, we define  the following operators.
\begin{equation}\label{boundp}
(t,v) \mapsto \left\{
  \begin{array}{rcll}

\mathcal Lv & = & \displaystyle \frac{4(n-1)}{n-2}\Delta_{g(t)}v-R_{g(t)}v+cv^{\frac{n+2}{n-2}}, & \mbox{ in}\quad M\\
\mathcal Bv & = & \displaystyle \frac{\partial_{g(t)}(\cdot)}{\partial \nu}+\frac{n-2}{2(n-1)}H_{g(t)}v, & \mbox{ in}\quad \Sigma.
\end{array}
  \right.
 \end{equation}
The next step is to apply the method of sub- and super-solutions as before to  the operator \eqref{boundp} in order to get a solution $\Phi(t)$ of \eqref{dir}. More specifically, we scale $u$ in such way that 
\begin{equation*}
\mathcal L(1+tu)\leq 0\leq \mathcal L(1-tu)\quad \mbox{and}\quad \mathcal B(1 - tu)\leq0\leq \mathcal B(1 +tu),
\end{equation*}
which implies the existence of a solution $\Phi(t)$ such that
$$ 1-tu\leq \Phi(t) \leq 1+tu$$ provided $t > 0$ is small enough. For $t<0$ we proceed in an analogous way.
Finally, the uniqueness of a solution to \eqref{deriv2} follows from standard  estimates as in Corollary \ref{derivative}. 

\medskip

Now we state in the following a variational characterization of critical points of volume and area of the boundary functional.

\begin{theorem}[Theorem \ref{teo1}]\label{generalteo1}
Let $(M,g)$ be a compact Riemannian manifold with non empty boundary $\Sigma$. Let $c\in\mathbb R$ be a fixed constant.  
\begin{enumerate}[(a)]

\item\label{item006}  Suppose  $g\in \mathcal{M}_{0,c}$ is a metric satisfying the property that the first eigenvalue of the problem  $\eqref{stek0}$ is positive, then 
$g$ is a critical point of $
g\mapsto \mbox{Area($\Sigma$,$g$)}
$ in $ \mathcal{M}_{0,c}$
if and only if there exists $V\in C^\infty(M)$ such that
\begin{equation}\label{eqespecial3}
\left\{
  \begin{array}{rccl}
-(\Delta_g V)g+\mbox{Hess}_gV-V\Ric_g& = & 0 & \mbox{in } M\\
\displaystyle\frac{\partial V}{\partial \nu}g-V\Pi_{g} & = & g & \mbox{on }\Sigma.
\end{array}
  \right.
 \end{equation} 
\item Suppose  $g\in \mathcal{M}_{c,0}$ is a metric satisfying the property that the first eigenvalue of $(n-1)\Delta_g + c$ with  Neumann boundary condition is positive, then 
$g$ is a critical point of $
g\mapsto \mbox{Vol(M,g)}
$
on $ \mathcal{M}_{c,0}$
if and only if there exists $V\in C^\infty(M)$ such that
\begin{equation*}\label{eqespecial2}
\left\{
  \begin{array}{rccl}
-(\Delta_g V)g+\mbox{Hess}_gV-V\Ric_g & = & g & \mbox{in } M\\
\displaystyle\frac{\partial V}{\partial \nu}g-V\Pi_{g} & = & 0 &\mbox{on } \Sigma.
\end{array}
  \right.
 \end{equation*} 
\end{enumerate}
\end{theorem}

\begin{proof}

By Proposition \ref{subman} there is a neighborhood $U$ of $g$ on the space of smooth metric such that $U\cap\mathcal M_{0,c}$ is a submanifold. Suppose that $g$ is a critical point of the area functional  $Area(\Sigma,\cdot)$ in  $\mathcal{M}_{0,c}$. Since the first eigenvalue of \eqref{stek0} is positive,  the Fredholm alternative (cf. \cite{Lie,Sa}) implies the existence of a unique  smooth function $V$ on $M$ satisfying:
\begin{equation}\label{fred2}
\left\{
  \begin{array}{rccl}
\Delta V&=& 0 & \mbox{in } M\\
\displaystyle\frac{\partial V}{\partial \nu}-\frac{c}{n-1}V&=& 1 & \mbox{on }\Sigma.
\end{array}
  \right.
 \end{equation} 

We will prove that $V$ satisfies \eqref{eqespecial3}. Let $h$ be a smooth symmetric (0,2)-tensor. For small $|t|$ we have that $g+th$ is a smooth metric in $M$.
By Proposition \ref{target}
there  exists $t_0>0$ and $\varepsilon>0$ such that for any $t\in(-t_0,t_0)$, there exists a unique positive solution $\Phi(t)$ of \eqref{neuscal} which is differentiable at $t = 0$ with $\Phi(0) = 1$. Thus, the metric $g(t)=\Phi(t)^{4/(n-2)}(g+th)$ is a $C^1$ curve in $\mathcal{M}_{0,c}$ such that $g(0)=g$ and
$$
\frac{d}{dt}\Big|_{t=0}\left(\Phi(t)^{4/(n-2)}g(t)\right)=\frac4{n-2} \hat\Phi(t)g+h,
$$
where $\hat\Phi:=\left.\dfrac{d}{dt}\right|_{t=0}\Phi$. This implies that
 \begin{equation}\label{inttrace}
  \int_{\Sigma}\left(\frac{4(n-1)}{n-2}\hat\Phi+\tr_{g}(h|_\Sigma)\right)da=0,
\end{equation}
 since $g$ is a critical point of the functional $g\mapsto\mbox{Area}(\Sigma,g)$ in $ \mathcal{M}_{0,c}.$ Furthermore, $\hat\Phi$ satisfies \eqref{deriv}. On the other hand, using  integration by parts, \eqref{fred2} and Corollary \ref{derivative} we obtain that
 $$\begin{array}{rcl}
  \displaystyle\frac{4(n-1)}{n-2}\int_{\Sigma} \hat\Phi  da & = & \displaystyle  \frac{4(n-1)}{n-2}\int_{\Sigma}  \Big(\frac{\partial V}{\partial \nu}\hat\Phi-\frac{c}{n-1}V\hat\Phi\Big) da \\
  \\
  & = & \displaystyle \frac{4(n-1)}{n-2}\int_{\Sigma}  \Big(\frac{\partial \hat\Phi}{\partial \nu}V-\frac{c}{n-1}V\hat\Phi\Big) da\\
  \\
  & & \displaystyle -\frac{4(n-1)}{n-2}\int_{M}  V\Delta \hat\Phi dv\\
  \\
  &  = & \displaystyle -2 \int_{\Sigma}V \delta H_g h da- \int_M V\delta R_ghdv.
 \end{array}$$

By \eqref{green} and \eqref{inttrace}, we get that 
$$
 \int_{ M}\langle h, -(\Delta_g V)g+\mbox{Hess}_gV-V\Ric_g  \rangle + \int_{\Sigma}\left\langle \frac{\partial V}{\partial \nu}g-V\Pi_{g}-g,h\right\rangle =0.
$$
Since $h$ is any $(0,2)$-tensor, we conclude that $V$ satisfies \eqref{eqespecial3}.

Now suppose that $V$ satisfies  equation \eqref{eqespecial3}. Let $h$ be a smooth symmetric $(0,2)$-tensor in the tangent space of $g$ in $\mathcal M_{0,c}$. This implies that $\delta R_gh=0$ and $\delta H_{g}h=0$. Therefore, by \eqref{green} we obtain
$$\begin{array}{rcl}
  0 & =  & \displaystyle 2 \int_{\Sigma}V \delta H_g h da+ \int_M V\delta R_ghdv  \\
  \\
  & = & \displaystyle \int_M\langle A_g^*V,h \rangle dv+\int_\Sigma\langle B_g^*V,h\rangle da\\
  \\
  & = & \displaystyle\int_\Sigma\tr_g (h|_\Sigma) da.
\end{array}$$
Therefore $g$ is a critical point to the area functional $g\mapsto\mbox{Area}(\Sigma,g)$ in $ \mathcal{M}_{0,c}.$ Hence item \eqref{item006} follows. 

 Assume now that  $g$ is a critical point of the volume functional on  $\mathcal{M}_{0,c}$. Since the first eigenvalue of $(n-1)\Delta_g+c$ with Neumann boundary condition is positive, the Fredholm alternative gives the existence of a smooth function $V$ such that
\begin{equation}\label{Ne}
\left\{
  \begin{array}{rccl}
\displaystyle\Delta V+\frac{c}{n-1}V & = & \displaystyle-\frac{n}{n-1} &\mbox{in } M\\
\vspace{-0,2cm}\\
\displaystyle\frac{\partial V}{\partial \nu} & = & 0  & \mbox{on }\Sigma.
\end{array}
  \right.
 \end{equation} 
 
 Let $h$ be a smooth symmetric (0,2)-tensor. 
By Proposition \ref{target2}, for each $t$ sufficiently small, it is possible to find a smooth positive $\Phi(t)$ on $M$ with $\Phi(0)=1$ and $\Phi(t)^{4/(n-2)}(g+th)\in \mathcal{M}_{c,0},$ where $\Phi$ is differentiable at $t=0$ and $\hat\Phi:=\left.\dfrac{d}{dt}\right|_{t=0}\Phi$ satisfies \eqref{deriv2}. Thus 
$$
\left.\frac{d}{dt}\right|_{t=0}\left(\Phi(t)^{4/(n-2)}g(t)\right)=\frac{4}{n-2} \hat\Phi g+h.
$$
Since $g$ is a critical point of the volume functional on $ \mathcal{M}_{c,0},$ we have  
 \begin{equation}\label{eq009}
  \int_{M}\left(\frac{4n}{n-2}\hat\Phi+\tr_{g}h\right)dv=0.
\end{equation}
 
  Using integration by parts, Proposition \ref{target2}  and  (\ref{Ne}) we have
$$\begin{array}{rcl}
     \displaystyle\frac{4n}{n-2}\int_{M} \hat\Phi  da & = & \displaystyle-\frac{4(n-1)}{n-2}\int_{M} \hat\Phi\left(\Delta V+\frac{c}{n-1}V\right)\Phi dv \\  
     \\
     & = & \displaystyle- \frac{4(n-1)}{n-2} \int_{M}V\left(\Delta\hat\Phi+\frac{c}{n-1}  \hat\Phi\right) dv 
     \\ 
     \\
     & & \displaystyle+\frac{4(n-1)}{n-2}\int_{\Sigma}V\frac{\partial \hat \Phi}{\partial\nu} da \\
     \\
     & = &\displaystyle -2 \int_{\Sigma}V \delta H_g h da- \int_M V\delta R_ghdv.
\end{array}$$

By \eqref{green} and \eqref{eq009}, we get that 
$$
 \int_{ M}\langle h, -(\Delta_g V)g+\mbox{Hess}_gV-V\Ric_g  -g\rangle + \int_{\Sigma}\left\langle \frac{\partial V}{\partial \nu}g-V\Pi_{g},h\right\rangle =0.
$$
Since $h$ is any $(0,2)$-tensor, we conclude that $V$ satisfies \eqref{eqespecial3}.

The conversely follows as in item \eqref{item006}. Therefore the theorem follows.
\end{proof}

\section{The  modified Kazdan-Warner-Kobayashi problem for manifolds with boundary}\label{prescsection}

In this section we begin by reviewing the notions needed in the sequel. Also we  collect  few technical result for the purpose of  proving  prescribing curvature results in Section \ref{Presc}.

\subsection{Yamabe invariant}\label{yaminva}
Let $M$ be a smooth compact manifold of dimension $n\geq 3$ with non empty boundary $\Sigma$. The total scalar curvature plus total mean curvature functional $F$ is defined in \eqref{functional}, which is defined on the space of Riemannian metrics in $M$. It is well known that $F(u^{\frac{4}{n-2}}g_0)=E(u)$, where
\begin{equation}\label{eq018}
    E(u)=\int_M|\nabla u|^2dv+\frac{n-2}{4(n-1)}\int_MR_{g_0}u^2dv+\frac{n-2}{2}\int_{\Sigma}H_{g_0}u^2da.
\end{equation}

Using this, given a conformal class $C=[g_0]$ of the Riemannian metric $g_0$, we obtain that the Yamabe constant \eqref{eq014} of $(M,C)$ can be written as
$$
    Y_\lambda(M,C)
    =\inf_{u\in H^1(M,g_0)\backslash\{0\}} I_\lambda(u),
$$
for $\lambda\in\{0,1\}$, where
$$I_\lambda(u)=\frac{E(u)}{\lambda\left(\int_Mu^{\frac{2n}{n-2}}\right)^{\frac{n-2}{n}}+(1-\lambda)\left(\int_{\Sigma}u^{\frac{2(n-1)}{n-2}}\right)^{\frac{n-2}{n-1}}}.$$

It is important to recall that if  $ Y_0(M,C)>0$ (resp. $ Y_0(M,C)=0$), then there exists a conformal metric with
zero  scalar curvature in $M$ and positive (resp. zero) mean curvature on $\Sigma.$

In the following, we recall well-known facts giving  conditions in which the infimum in \eqref{eq014} is achieved. This infimum is known as a Yamabe metric, that is, a metric of constant scalar curvature and minimal boundary, in the case $\lambda=1$, or scalar flat metric with constant mean curvature, in the case $\lambda=0$. In fact, the next theorem summarizes the recent contributions to the problem started by J. F. Escobar \cite{E2,E1} and studied also by S. M. Almaraz \cite{almaraz}, S. Brendle and S.-Y. S. Chen \cite{Brendle-Chen},  F. C. Marques \cite{Marques2,Marques1} and M. Mayer and C. B. Ndiaye \cite{Mayer-Ndiaye}.

\begin{theorem}\label{fact 1}
Let $M$ be a smooth compact manifold with non empty boun-dary and dimension $n\geq 3$. If $C$ is a conformal class in $M$ such that
\begin{enumerate}[(a)]
\item\label{item004} $-\infty< Y_1(M,C)<Y_1(\mathbb{S}^n_+,\partial \mathbb{S}^n_+),$ then there  exists a metric $g$ on $C$  with  zero mean curvature  and constant scalar curvature 
$$
R_{{g}}=\frac{4(n-1)}{n-2} Y_1(M,C) \mbox{Vol}(M, {g})^{-2 / n}.$$
\item  $-\infty< Y_0(M,C)<Y_0(\mathbb{B}^n_+,\partial \mathbb{B}^n_+),$ then there  exists a scalar flat metric $\bar g$ on $C$  with  constant mean curvature   equal to 
$$
H_{\bar{g}}=\frac{2}{n-2}  Y_0(M,C) \mbox{Area}(M, \bar{g})^{-\frac{1}{n-1}}.
$$
\end{enumerate}
\end{theorem}

Another useful fact in this paper is the following lemma that follows from continuity property  (see Appendix for a proof) of the Yamabe constant.

\begin{lemma}\label{fact 3}
Let $M$ be a smooth compact manifold with non empty boundary and dimension $n\geq 3$. For $\lambda\in\{0,1\}$ fixed, given a constant $c<\sigma_\lambda(M)$, there exist a conformal class $C$ such that $Y_\lambda(M,C)=c$.
\end{lemma}

We say that a metric $g$ realizes the Yamabe invariant if $ Y_\lambda(M,C)=\sigma_\lambda(M)$, where $C$ is the conformal class of $g$.
Using variational  arguments we can prove that a metric $g$ with unit volume (resp. unit area of the boun-dary)  which realizes the Yamabe invariant is Einstein with totally geodesic boundary (resp. scalar flat with umbilical boundary). 

\begin{proposition}\label{Einstein}
Suppose that $\sigma_1(M)$ is negative. Then any  metric $g$ with $\mbox{Vol}(M,g)=1$ which realizes $\sigma_1(M)$ is Einstein with totally geodesic boundary.
\end{proposition}

\begin{proof}
Suppose that the metric $g$ has unit volume and realizes the Ya-mabe invariant. Consider a family of Riemannian metrics $\{g(r)\}$, with $r\in(-\varepsilon,\varepsilon)$, given by $g(r)=g+rh$, for $\varepsilon>0$ small enough. By the definition of the Yamabe invariant \eqref{eq015}, we have $Y_1(M,C_r)\leq\sigma_1(M)<0$ for all $r\in(-\epsilon,\epsilon)$, where $C_r=[g(r)]$.
From the resolution of the Yamabe problem on manifold with boundary there
 exists a unique positive function  $u_r>0$ such that  $\tilde{g}(r)=u_r^{\frac{4}{n-3}}g(r)$ has constant scalar curvature equal
to $Y_1(M,C_r)<0$ and zero mean curvature on the boundary for all $r\in(-\varepsilon,\varepsilon)$. This implies that
$\left.\frac{\partial}{\partial r}\right|_{t=0}Y_1(M,C_r)=0$, since $\sigma_1(M)$ is maximum to the function
$r\mapsto Y_1(M,C_r)$.

Using the variation of the scalar curvature and the mean curvature, see Section \ref{preli}, a standard computation (see \cite{A}) gives

\begin{eqnarray*}\nonumber
0=\left.\frac{\partial}{\partial r}\right|_{r=0}Y_1(M,C_r)&=&-\int_M\left\langle{\rm Ric}_g-\frac{R_g}{2}g,h\right\rangle\,dv\\
& &-\int_{\Sigma}\left\langle\Pi_g-H_g,h\right\rangle\,da,
\end{eqnarray*}
for any trace-free symmetric (0,2)-tensor $h.$ 
By  the method of Lagrange Multipliers to the above equation, there exists $c\in\mathbb R$ such that  $\Pi_g=H_gg$ and 
${\rm Ric}_g-\frac{R_g}{2}g=\frac{c }{2}g.
$  Thus, one may easily check that $M$ is an Einstein metric with
totally geodesic boundary.
\end{proof}

In a similar way we have (which we do not prove due to the similarity of the arguments). 

\begin{proposition}\label{Umbilical}
Suppose that $\sigma_0(M)$ is negative. Then any  metric $g$ with $\mbox{Area}(M,g)=1$ which realizes $\sigma_0(M)$ is Ricci flat with umbilical boundary.
\end{proposition}

We hope the same consequence of the previous propositions for any metric $g$ with unit volume (resp. unit area of boundary)  realizing $\sigma_1(M)$ (resp. $\sigma_0(M)$). Unfortunately this is not clear for $\sigma_1(M)>0$ (resp.  $\sigma_0(M)>0$).

\subsection{Prescribing theorems}\label{Presc}

In order to prove our  prescribing curvature results, first we observe that part of techniques used in beginning  of this subsection  has been introduced in \cite{CV}. Thus, in order to help the reader and for the sake of clarity, we recall, when needed, the  appropriately modified statements of the theorems we use.

We are interested in to prescribe the scalar curvature and mean curvature on the sets $\mathcal{M}^{2,p}\cap\{g; \mbox{ Vol}(M,g)=1\}$ and  $\mathcal{M}^{2,p}\cap\{g; \mbox{ Area}(\Sigma,g)=1\}$. To this end we will consider the map $\Psi(g)=(R(g), 2H(g))$ under these constraints on the volume and on the area of the boundary. In this case our key tool is  the implicit function theorem, which allow us to locally solve, in an appropriate topology, the following equation (an equation for the metric $g$): $\Psi(g)=(f_1,2f_2).$

From now one we define $\mathcal{S}_g$ as the linearization of $\Psi$ and denote by $\mathcal S_g^*$ its formal $L^2$-adjoint. We have the following theorem which is an immediate consequence of  Theorem 3.5  of \cite{CV}.

\begin{theorem}\label{implicit}
 Let $(M^n,g_0)$ be a compact Riemannian manifold with non empty boundary $\Sigma$ and dimension $n\geq 2$. Let $f=\left(f_{1}, f_{2}\right) \in L^{p}(M) \oplus W^{\frac{1}{2}, p}(\Sigma)$ with  $p>n$.
\begin{enumerate}[(a)]
    \item Suppose that $\mbox{Vol}(M,g_0)=1$, $f_{2}=H_{g_{0}}$ and
$S_{g_0}^{*}$ is injective. There exists $\eta>0$ such that if $\left\|f_{1}-R_{g_{0}}\right\|_{L^{p}}<\eta,$
then there is a metric $g_{1} \in \mathcal{M}^{2, p}$ with unit volume such that $\Psi\left(g_{1}\right)=f.$
 \item Suppose that $\mbox{Area}(\Sigma,g_0)=1$, $f_{1}=R_{g_{0}}$  and
$S_{g_0}^{*}$ is injective. There exists $\eta>0$ such that if
$\left\|f_{2}-H_{g_{0}}\right\|_{W^{1 / 2, p}}<\eta,$
then there is a metric $g_{1} \in \mathcal{M}^{2, p}$ with unit area of the boundary such that $\Psi\left(g_{1}\right)=f.$
\end{enumerate}
 Moreover, $g_1$ is smooth
in any open set where $f$ is smooth.
\end{theorem}

We remark that a slight modification of  Proposition 3.1 and Proposition 3.3 of \cite{CV} gives conditions to the injectivity of $S_{g_0}^{*}$.
\begin{itemize}
    \item Assume $R_{g_0}=0$ and $\mbox{Vol}(g_0)=1$, then $S_{g_0}^{*}$ is injective if either  $\frac{\partial}{\partial \nu}-\frac{c}{n-1}$ has positive Steklov  spectrum or $H_{g_0} = 0,$ but $\Pi_{g_0}$ is not identically zero (Proposition 3.1).
    \item Assume $H_{g_0}=0$ and $\mbox{Vol}(g_0)=1,$ then  $S_{g_0}^{*}$ is injective if  either $\Delta_{g_0} + c/(n-1)$ has positive Neumann spectrum or  $R_{g_0} = 0,$ but $\mbox{Ric}_{g_0}$ is not identically zero (Proposition 3.3).
\end{itemize}

Next we state an approximate lemma proved in \cite{CV}   which says how to  approximate a function arbitrarily closely in $L^p(M)$ and $W^{\frac{1}{2},p}(\Sigma)$  if and only if these functions satisfy a suitable range condition.

\begin{lemma}[Approximation Lemma \cite{CV}]\label{appr} Let $(M,g)$ be a Riemannian\linebreak manifold with non empty boundary $\Sigma$ and dimension $n\geq2.$ 
\begin{itemize}
\item[(a)] Let $f, h\in C^{\infty}(\partial M).$ If the range of $h$ is in the range of $f$, that is, $\min f\leq h(x)\leq\max f $ on $\partial M$, then given any positive $\varepsilon$ there is a diffeomorphism  $\varphi$ of $M$ such that, for $p>2n,$ we have  that
$$\|f\circ \varphi-h\|_ {W^{\frac{1}{2},p}(\Sigma)}<\varepsilon.$$
\item[(b)] Let $f, h\in C^{\infty}( M).$ If   the range of $h$ is in the range of $f$, that is, $\min f\leq h(x)\leq\max f $ on $ M$, then given any positive $\varepsilon$ there is a diffeomorphism  $\varphi$ of $M$ such that, for $p>n,$ we have  that 
$$\|f\circ \varphi-h\|_ {L^{p}( M)}<\varepsilon.$$
\end{itemize} 
\end{lemma}

Arguing as Proposition 5.1 of \cite{CV},  Approximation Lemma \ref{appr} and Theorem \ref{implicit} we obtain  the following result.

\begin{proposition}\label{min-max} Let $(M,g_0)$ be a Riemannian manifold with non empty boundary $\Sigma$ and dimension $n\geq2.$ 
\begin{enumerate}[(a)] 
\item Assume that $g_0$ is  scalar flat with $\mbox{Area}(\Sigma, g_0)=1$. Let $f\in C^\infty(\Sigma)$ satisfying  
$\min f < H_{g_0}<\max f.$ Then there exists a scalar flat Riemannian metric $g$ such  that $H_{g}=f$ and $\mbox{Area}(\Sigma, g)=1.$
\item Assume that $g_0$ has minimal boundary, unit volume. Let  $f\in C^\infty(M)$ satisfying  $\min f <R_{g_0}<\max f.$  Then there exists a Riemannian  metric $g$ with minimal boundary such  that $R_{g}=f$ and $\mbox{Vol}( M, g)=1.$
\end{enumerate}
\end{proposition} 

\begin{proof}
By completeness, we prove  item (a) (item (b) is entirely analogous). 
Assume $Ker ~\mathcal{S}_{g_0}^*=0$.  By Lemma \ref{appr}, given $\eta>0$, there  exists a diffeomorphism  $\varphi$ of $M$ such that
$
\| f\circ\varphi-H_{g_0}\|_{W^{1/2,p}(\partial M)}<\eta,
$ for $p>2n$.
According to Theorem \ref{implicit} there is a metric $g_1$ with unit area satisfying $$\Psi(g_1)=(0,f\circ\varphi).$$ 
The  diffeomorphism invariance of scalar  curvature and the mean curvature of the boundary\footnote{The diffeomorphism invariance of the scalar curvature is clear. In the case of the mean curvature, it is not difficult to prove that for any diffeomorphism $\varphi: M \rightarrow M$ 
 such that $\varphi(\partial M)=\partial M,$ we have $H_{\varphi^{*} g}=\varphi^{*} H_{g}$ on $\Sigma$} implies that the required metric  is  $g=(\varphi^{-1})^*(g_1)$.

If $Ker ~\mathcal{S}_{g_0}^*\neq 0$, which by item \eqref{item001} of  Proposition \ref{properties} implies that $H_{g_0}$ is constant,  we  
perturb $g_0$ slightly in order to have a scalar flat metric $g_1$ with  non constant mean curvature $H_{g_1}$  still satisfying  
$\min f < H_{g_1}< \max f.$ 
\end{proof}

\subsection{Proof of Theorem \ref{dim}, \ref{prescribing thm} and \ref{prescribing thm2}} 

Now we present the proof  of Item \eqref{item016} of Theorem \ref{dim}.

\begin{proof}[Proof of item \eqref{item016} of Theorem \ref{dim}]

Since for dimension $n=2$ the first part of the Theorem is clear. It remains to prove the other implication. 
Assume that $f$ is a constant equal to $c=2\pi\chi(M).$ Then by Theorem 1.1 of \cite{CV}, there is a flat metric with constant geodesic curvature equal to $c$. By Gauss Bonnet Theorem, the boundary has unit length when $M$ carry this metric.  
Suppose that $\min f<2\pi\chi(M)<\max f$. Again by  Theorem 1.1 of \cite{CV}, there is a flat metric with constant geodesic curvature equal to $2\pi\chi(M)$. Using  Proposition \ref{min-max} we conclude our result.
\end{proof}

 We left to the reader the proof of item \eqref{item018} of  Theorem \ref{dim}  which is  completely identical. For dimension $n\geq 3$ we need of the following lemma, which is a consequence of Lemma 4.1 of \cite{S1} and from the definition of $Y_\lambda(M,C)$ in \eqref{eq014}.
\begin{lemma}\label{supinf}
  Let $(M ,g_0)$ be a compact Riemannian manifold with non emp-ty boundary $\Sigma$ and dimension  $n\geq 3$.
\begin{enumerate}[(a)]
\item Suppose that $-\infty< Y_1(M,C)\leq 0.$ Then the scalar curvature $R_g$ of any metric $g\in [g_0]$ with minimal boundary satisfies
\begin{equation}\label{eq003}
    (\min R_g)\mbox{Vol}(M,g)^{\frac{2}{n}}\leq \frac{4(n-1)}{n-2} Y_1(M,C)\leq (\max R_g)\mbox{Vol}(M,g)^{\frac{2}{n}}.
\end{equation}
If the equality holds, $R_g$ is constant.
\item  Suppose that  $-\infty< Y_0(M,C)\leq 0.$ Then the mean curvature of any metric $g\in[g_0]$ with zero scalar curvature satisfies
 $$
(\min H_g)\mbox{Area}(\Sigma,g)^{\frac{1}{n-1}}\leq\frac{2}{n-2}  Y_0(M,C)\leq (\max H_g)\mbox{Area}(\Sigma,g)^{\frac{1}{n-1}}.
$$	
If the equality holds, $H_g$ is constant.
\end{enumerate}

\end{lemma}

Now we are ready to prove the Theorem \ref{prescribing thm}.

\begin{proof}[Proof of  Theorem \ref{prescribing thm}]
Let us prove the item \eqref{item002}.

First, suppose $f\in C^\infty(M)$ is the scalar curvature of some metric $M$ with minimal boundary and unit volume. Since $\sigma_1(M)\leq 0$, the assumption implies that $-\infty<  Y_1(M,C)\leq 0$ for some conformal class $C$. The Lemma \ref{supinf} implies the inequality \eqref{eq003} for any metric $g\in C$. In particular, if $\mbox{Vol}(M,g)=1$, we get \eqref{eq004}.

Now suppose that $f\in C^\infty(M)$ satisfies the inequality \eqref{eq004}. We note that for any constant $c$ with $\frac{n-2}{4(n-1)}c<\sigma_1(M), $  by Theorem \ref{fact 1} and Lemma \ref{fact 3}, there is a metric $g$ such that 
$$
R_g=\frac{c}{\mbox{Vol}(M,g)^{2/n}}
$$
and zero mean curvature.
Observe that $\tilde{g}=\mbox{Vol}(M,g)^{-2/n}g$ is a metric with $R_{\tilde{g}}=c$,
$\mbox{Vol}(M,\tilde g)=1$  and zero mean curvature. 

Thus, if $f$ is constant, we are done. Otherwise, we can pick a constant $c\in\mathbb{R}$ such that 
$$\min f<c<\min\left\{\frac{4(n-1)}{n-2}\sigma_1(M),\max f\right\}.$$ 
Arguing as before, we can find a metric $g$ with unit volume, zero mean curvature and $R_g=c$. Finally, by Proposition \ref{min-max} we find a metric $g$ with $f$ as scalar curvature, minimal boundary and unit volume.

\medskip

Now let us prove item \eqref{item003}. In this case, we conclude by Theorem \ref{fact 1} the existence of a metric $g\in C$ such that
$$
\sigma_1(M)= Y_1(M,C)= \frac{n-2}{4(n-1)} R_g\mbox{Vol}(M,g)^{\frac{2}{n}},$$
which solves the problem if the function $f$ is constant. If $f$ is not constant, the argument is similar to the previous one.

Now, by Proposition \ref{Einstein} we obtain the conclusion of the theorem.

\end{proof}

The case where the Yamabe invariant $\sigma_0(M)$ is non positive is the content of Theorem \ref{prescribing thm2} which the prove is entirely analogous to the previous one. In this way, we omitted it.

\subsection{Positive Yamabe invariant and prescribing curvature results}

The next result is an obvious  modifications of  Gluing Lemma 4.1  that appears in Kobayashi \cite{K}. Its proof  is quite similar and we will omit.

\begin{lemma}\label{modifikob}
 Let $(M_{1}, g_{1})$ be a  Riemannian manifolds with non empty boun-dary of  dimension $n \geqq 3$ and $(M_2, g_{2})$ be a closed manifold with the same dimension. Assume that  
$
R_{g_{1}}\left(p_{1}\right)=R_{g_{2}}\left(p_{2}\right)=n(n-1)
$
 at some points $p_1 \in \mbox{int }M_{1}$ and $p_2 \in M_{2}$. Then for any $\varepsilon>0,$ there is a metric g of $M_{1} \# M_{2}$ such that 
\[
\left|\operatorname{Vol}\left(M_{1} \# M_{2}, g\right)-\sum_{i=1}^{2} \operatorname{Vol}\left(M_{i}, g_{i}\right)\right|<\varepsilon.
\]
Moreover,  there are isometric imbeddings $$\varphi_{i}:\left(M_{i} \backslash B_{i}, g_{i}\right) \rightarrow\left(M_{1} \# M_{2}, g\right), i=1,2,$$
such that $|R_{g}(x)-n(n-1)|<\varepsilon$ for $x \in M_{1} \# M_{2} \backslash\left(\operatorname{Im} \varphi_{1} \cup \operatorname{Im} \varphi_{2}\right).$ Here each $B_{i}$ is a small ball containing $p_{i} \in M_{i}$.

 \end{lemma}

We apply the above lemma  to prove Theorem \ref{positive}, which needs of a careful control on scalar curvature and volume.

\begin{proof}[Proof of Theorem \ref{positive}]
Since $\sigma_1(M)>0$, then by item \eqref{item004} of Theorem \ref{fact 1} and Lemma \ref{fact 3}, for all $v\in\left(0,\left(\frac{4}{n(n-2)}\sigma_1(M)\right)^{\frac{n}{2}}\right]$ there exists a metric $g$ with minimal boundary such that $R_g=n(n-1)$ and $\mbox{Vol}(M,g)=v$. Also we observe that given $a\in(0,\mbox{Vol}(\mathbb S^n)]$ there exists a metric in the closed sphere with volume equal to $a$ and constant scalar curvature equal to $n(n-1)$ (see \cite{K}, for instance). Using Lemma \ref{modifikob} we can perform a connected sum of $M$ with $\mathbb S
^n$ to conclude that given $v$ and $a$ as before, for any $\varepsilon>0$ there exists a metric $g$ in $M$ with minimal boundary such that for $|R_g-n(n-1)|<\varepsilon$ and $|\mbox{Vol}(M,g)-(v+a)|<\varepsilon$. More generally, if we perform arbitrary connected sums, we obtain that given $v>0$, for any $\varepsilon>0$ there is a metric $g$ in $M$ with minimal boundary such that $|R_g-n(n-1)|<\varepsilon$ and $|\mbox{Vol}(M,g)-v|<\varepsilon$. By rescaling, this implies that given $r>0$, for all $\varepsilon>0$ there exists a metric $g$ with unit volume, minimal boundary and $|R_g-r|<\varepsilon$. Besides, if $r\leq \sigma_1(M)$ is a constant, by Theorem 5.1 and Lemma \ref{fact 3}, we conclude that there exists a metric $g$ with unit volume, minimal boundary and $R_g=r$.

Therefore, if $f\in C^\infty(M)$ is not constant we can choose a constant $r$ such that $\min f<r<\max f$ and then by Proposition \ref{min-max} the result follows.

\end{proof}

 The remaining case of the  modified Kazdan-Warner-Kobayashi problem with $\sigma_1(M)>0,$  namely, 
the problem of prescribing a constant larger than or equal to $\sigma_1(M),$ has an additional obstruction to its resolution. In fact, in analogy with the Matsuo's proof \cite{M},  it would be necessary to prove  a localized gluing result  for constant scalar curvature metrics with minimal boundary in which the total volume is preserved (analogue of Theorem 1.6 of \cite{CMT}).

Before to prove Theorem \ref{finalthm} we need of some preliminaries results.
Let $(\mathbb S^{n}, g_{n}),$ $n \geq 3$, be the Euclidean unit sphere endowed with the metric $g_{n}=d r^{2}+(\sin ^{2}r) g_{n-1},$ where $r$ is the intrinsic distance relative to $g_{n}$ from the north pole and $g_{n-1}$ is the standard metric on the unit $n-1$ sphere. For an interval $I \subseteq[0, \pi]$ define $A(I)=\left\{x \in \mathbb S^{n}: r(x) \in I\right\}$.

\begin{lemma}[Lemma 3.1 of \cite{K}]\label{lemma3.1}
 For any $\varepsilon_1>0$ and $0<\varepsilon_{2}<\pi$ there exists a positive function $f=f(r)$ on $\mathbb S^{n}$ such that the metric $g=f^{-2} g_{n}$ satisfies
 \begin{enumerate}[(a)]
     \item\label{item010} $\left|R_{g}-n(n-1)\right|<\varepsilon_{1};$
     \item$\left| \mbox{Vol}\left(\mathbb S^{n}, g\right)-2 \mbox{Vol} \left(\mathbb S^{n}, g_{n}\right) \right|<\varepsilon_{1}$;
     \item\label{item008} $f(r)=1$ for $r>\varepsilon_{2}$ and $(A([0,\varepsilon_3)),g)$ is isometric to $(A((\varepsilon_3,\pi]),g)=(A((\varepsilon_3,\pi]),g_n)$ for some $\varepsilon_3<\varepsilon_2$;
     \item\label{item009} $ 0<f(r) \leq 1$ and $|df(r)/dr| \leq 2 / \sin (r)$ for all $r.$ 
 \end{enumerate}

\end{lemma}

We observe that the obtained metric is close to a singular metric isometric to $4 d r^{2}+\sin^{2}(2r)g_{n-1}, 0 \leq r \leq \pi,$ 
 as long as  $\varepsilon_{1}$ tends to $0.$ Moreover, away of the  part $r=\pi/2,$ this  metrics has constant sectional curvature 1 and volume equal to a twice the volume of the unit $n$-sphere.

We also recall the following results which is the content of Corollary 3.6 of O. Kobayashi \cite{K}.

\begin{proposition}\label{prop001}
If $R_g(p)=n(n-1)$ for some $p\in M$, then for any $\varepsilon>0$ there is a function $f$ equal to zero outside of a ball centered at $p$ and radius $0<\varepsilon_1\leq\varepsilon$, such that the metric $\overline g=e^fg$ satisfies the conditions
\begin{enumerate}[(a)]
    \item $|R_{\overline g}-R_g|\leq \varepsilon$;
    \item $|\mbox{Vol}(M,\overline g)-\mbox{Vol}(M,g)|<\varepsilon$;
    \item $Ric_{\overline g}(p)=n(n-1)\overline g$.
\end{enumerate}
\end{proposition}

Now we present the proof of Theorem \ref{finalthm}.

\begin{proof}[Proof of Theorem \ref{finalthm}]

Here we proceed as Kobayashi’s Theorem 4 of \cite{K}. 

First of all, we note that it is enough to prove that for any $\varepsilon>0$ and $k\geq 0$ an integer there exists a metric $g\in C$ such that 
\begin{equation}\label{eq010}
    |R_g-n(n-1)|<\varepsilon,\;\;\;\;\;\;\quad H_{g_0}\equiv 0
\end{equation}
and
\begin{eqnarray}\label{eq011}
    \left|\mbox{Vol}(M, g)-\left(\frac{Y_1(M,C)}{n(n-1)}\right)^{n/2}- 2k\mbox{Vol}(\mathbb S_+^n)\right|<\varepsilon.
\end{eqnarray}

In fact, if this is the case, consider the metric $g_0=Vol(M,g)^{-2/n}g$. Thus, we have $Vol(M,g_0)=1$, $H_{g_0}\equiv 0$ and $R_{g_0}=Vol(M,g)^{2/n}R_g$. Therefore, since $\sigma_1(\mathbb S^n_+)=n(n-1)\mbox{Vol}(\mathbb S^n_+)^{2/n}$, we find
\begin{align*}
     \left|R_{g_0}-(Y_1(M,C)^{n/2}+ k(2\sigma_1(\mathbb S_+^n))^{n/2})^{2/n}\right|  \leq  Vol(M,g)^{2/n}|R_g-n(n-1)|\\
          \displaystyle +n(n-1)\left|\mbox{Vol}(M, g)^{2/n}-\left(\left(\frac{Y_1(M,C)}{n(n-1)}\right)^{n/2}+ 2k\mbox{Vol}(\mathbb S_+^n)\right)^{2/n}\right|\leq\varepsilon_1,
\end{align*}
where $\varepsilon_1$ depends on $\varepsilon$ and goes to zero as $\varepsilon\rightarrow 0$.

The prove of the existence of a metric satisfying \eqref{eq010} and \eqref{eq011} is by induction on $k$. For $k=0$, the result follows by the existence of a metric $g\in C$ with $R_g=\frac{4(n-1)}{n-2}\mbox{Vol}(M,g)^{-2/n}Y_1(M,C)$, see Theorem \ref{fact 1}. Now suppose that we can find such metric for $k-1$ and let us prove for $k$. Denote by $v$ the quantity $\left(\frac{Y_1(M,C)}{n(n-1)}\right)^{n/2}+ 2(k-1)\mbox{Vol}(\mathbb S_+^n)$. 

By rescaling we can suppose that $R_g=n(n-1)$ at some point $p_0\in\mbox{Int} ~M$. Applying Proposition \ref{prop001}, given $\varepsilon>0$ we find a metric $\tilde g$ which coincides with $g$ outside of a ball centered at $p_0$ and radius small then $\varepsilon$, thus with minimal boundary, such that $|R_{\tilde g}-n(n-1)|<\varepsilon$, $|\mbox{Vol}(M,\tilde g)-v|\leq \varepsilon$ and $Ric_{\tilde g}(p_0)=n(n-1)\tilde g$.

Let $f=f(r)$ be the positive function given by Lemma \ref{lemma3.1} with small $\varepsilon_2>0$. Consider the distant function $r=d(x,p_0)$ in $M$ and regard the function $f$ as a smooth positive function in $M$. Note that $f(r)=1$ for all $r>\varepsilon_2$.

One recall the following formula for the Laplacian that follows by a straightforward computation
$$
\Delta_{\tilde g} f=\frac{d^2f}{dr^2}+\left(\frac{n-1}{r}-\frac{r}{3} \mbox{Ric}_{\tilde g}(\nabla r, \nabla r)+O(r^2)\right) \frac{df}{dr},
$$
in a neighborhood of $p_0 \in M. $ Thus, since $R_{\tilde g}(p_0)=n(n-1)=R_{g_n},$ using items  \eqref{item008} and \eqref{item009} of Lemma \ref{lemma3.1}, for $0<r\leq\varepsilon_2$ we obtain that
\begin{equation}\label{diflapl}
\left|\Delta_{\tilde g} f-\Delta_{g_n} f\right| \leq K \frac{\varepsilon_2^2}{\sin \varepsilon_2},
\end{equation}
where  $K$ is a positive constant and $\Delta_{g_n}$ is the Laplacian of $g_n$.
 Consider the metrics $\overline{g}=f^{-2} \tilde g$ in $M$ and $g'=f^{-2} g_n$ in $\mathbb S^n$. Note that the metric $\overline g$ has minimal boundary for small $\varepsilon_2>0$. Using the well-known transformation law for the scalar curvature under conformal deformations
of the metric  we obtain the following expression.
$$
R_{\overline g}=2(n-1) f\left(\Delta_{\tilde g} f-\Delta_{g_n} f\right)+\left(R_{\tilde g}-n(n-1)\right) f^{2}+R_{g^{\prime}}.
$$
From items \eqref{item010} and 
\eqref{item008} of Lemma \ref{lemma3.1}, $\eqref{diflapl}$ and the assumption on the metric $\tilde g$ we obtain
$$
\left|R_{\overline g}-n(n-1)\right| \leq 2(n-1) K \frac{\varepsilon_2^2}{\sin \varepsilon_2}+\varepsilon+\varepsilon_{1}.
$$

Since $f(r)\equiv 1$ for all $r>\varepsilon_2$, then the volume satisfies
$\mbox{Vol}(M,\overline g)=\mbox{Vol}(B(p_0,\varepsilon_2),\overline g)+\mbox{Vol}(M\backslash B(p_0,\varepsilon_2),\tilde g)$ and $\mbox{Vol}(S^n,g^\prime)=\mbox{Vol}(B(p_n,\varepsilon_2),g^\prime)+\mbox{Vol}(M\backslash B(p_n,\varepsilon_2),g_n)$. Thus we get that
\begin{align*}
    \left|\mbox{Vol}(M, \bar{g})-\left(v+2\mbox{Vol}\left(\mathbb S_+^{n},g_n\right)\right)\right| & \leq |\mbox{Vol}(M\backslash B(p_0,\varepsilon_2),\tilde g)-v|\\
    & +|\mbox{Vol}(B(p_0,\varepsilon_2),\overline g)-\mbox{Vol}(B(p_n,\varepsilon_2),g^\prime)|\\
    & +|\mbox{Vol}(\mathbb S^n,g^\prime)-2\mbox{Vol}\left(\mathbb S^{n},g_n\right)|\\
    & +|\mbox{Vol}(\mathbb S^n\backslash B(p_n,\varepsilon_2),g_n)-\mbox{Vol}(\mbox{S}^n,g_n)|,
\end{align*}
since $g^\prime=g_n$ and $\overline g=\tilde g$ outside of $B(p_n,\varepsilon_2)$ and $B(p_0,\varepsilon_2)$, respectively. By Lemma \ref{lemma3.1} we have to estimate only the second term in the right hand side above. Note that
\begin{align*}
    \mbox{Vol}(B(p_0,\varepsilon_2),\overline g)-\mbox{Vol}(B(p_n,\varepsilon_2),g^\prime) & =\int_{B(0,\varepsilon_2)}f^{-n}(\sqrt{\det{\tilde g}}-\sqrt{\det{g_n }})dx\\
    & \leq c(n) \varepsilon_2^2\int_{B(0,\varepsilon_2)}f^{-n}dx\\
    & \leq d(n)\varepsilon_2^2\int_{B(0,\varepsilon_2)}f^{-n}\sqrt{\det g_n}dx\\
    & =d(n)\varepsilon_2^2\mbox{Vol}(B(p_n,\varepsilon_2),g^\prime),
\end{align*}
where $c(n)$ and $d(n)$ are positive constants which depends only on $n$.
Here we used that $\det \tilde g$ and $\det g_n$ are equal to $1+O(|x|^2)$ in normal coordinates. But since $\mbox{Vol}(B(p_n,\varepsilon_2),g^\prime)$ is closed to $\mbox{Vol}(\mathbb S^n,g_n)$, we conclude that 
$$\left|\mbox{Vol}(M, \bar{g})-\left(v+2\mbox{Vol}\left(\mathbb S_+^{n},g_n\right)\right)\right|$$
can be made so small as we want just taking $\varepsilon$, $\varepsilon_1$ and $\varepsilon_2$  small enough.

\end{proof}

\section{Appendix}\label{appe}

In this appendix,  we prove an analogous result of the continuity property of the Yamabe constant due to B\'erard Bergery \cite{BB}. This kind of fact is  well known, because its proof is similar  to the one that the eigenvalues $\lambda_j(g)$ of the Laplacian $\Delta_g$ depend continuously on $g$, but since we had difficulties in finding a reference for the appropriate version, we proved below. Before,  let us introduce some terminology.   
 
 We define a distance  $d''_x$ on the set of all symmetric positive definite symmetric tensors $P_x$, $x\in M,$ by
 $$
d''_x(\psi,\varphi)=\sup_{v\in T_xM\setminus\{0\}}\inf\left\{\delta>0; \exp(-\delta)<\frac{\varphi(v,v)}{\psi(v,v)}<\exp(\delta)\right\}, $$
 for $\psi$, $\varphi$ in $P_x$. It is not difficult to prove that $(P_x,d_x)$ is a complete space, see \cite{B}[Lemma 1.1]. On the space of smooth Riemannian metrics on $M$ we define  the distance $d$ on $\mathcal{M}$ by 
 $$
d(g_1,g_2)=d'(g_1,g_2)+d''(g_1,g_2), 
 $$
 where $d'(g_1,g_2)=|g_1-g_2|$ is the Fr\'echet norm\footnote{The Fr\'echet norm $|\cdot|$ on $S^2(M)$ is given by $|h|=\sum_{k=0}^\infty|h|_k(1+|h|_k)^{-1}$, where $|\cdot|_k$ is the standard $C^k$-norm on $S^2(M)$.} on $S^2(M)$ and $d''(g_1,g_2)=\sup_{x\in M}d''_x((g_1)_x,(g_2)_x)$.

 \begin{proposition}
Let $g_i$ be a sequence of smooth metrics, $R_{g_i}$ the scalar curvature of $g_i$, $H_{g_i}$ the mean curvature of $g_i|_{\partial M}$ and $C_i$ be a conformal class of $g_i$ and . Assume that 

\begin{equation*}
\left\{
  \begin{array}{rlcl}
 g_i & \to & g  & \mbox{ in the $C^0$-topology on $M$, }\\
 R_{g_i}&\to & R_g  & \mbox{ in the $C^0$-topology on $M$, and}\\
 H_{g_i}&\to & H_g  & \mbox{ in the $C^0$-topology on $\Sigma$}.
 \end{array}
  \right.
 \end{equation*} 
Then 
 $$\lim Y_\lambda(M,g_i,C_i)=Y_\lambda(M,g,C),$$
 for each $\lambda=0,1$.
 \end{proposition}
 \begin{proof}
  
We will prove only for $\lambda=1$. Given $\delta>0,$ assume that $g',g\in \mathcal{M}$ are metrics with $d(g,g')<\delta$. Thus, at each point $x\in M,$ it
holds that  $d''_x(g_x,g'_x) < \delta,$ which implies that
$$
\exp(-\delta)g_x'<g_x<\exp(\delta)g_x'.
$$
If we consider local coordinates $\{x_1,\ldots, x_n\}$  on an 
open set $U$ of $M,$ we have   
 $$
\exp\left(-\frac{n}{2}\delta\right)(\mbox{det}(g_{ij}'))^{\frac{1}{2}}<\mbox{det}(g_{ij}))^{\frac{1}{2}}<\exp\left(\frac{n}{2}\delta\right)(\mbox{det}(g_{ij}'))^{\frac{1}{2}}
$$ 
and 
$$
\exp(-\delta)(g^{ij})'<g^{ij}<\exp(\delta)(g^{ij})'.
$$
% where $g_{ij}=g(\partial/\partial_{x_i},\partial/\partial_{x_j})$ and $g'_{ij}=g'(\partial/\partial_{x_i},\partial/\partial_{x_j})$.
 
Thus, 
\begin{equation}\label{function}
\exp\left(-\frac{n}{2}\delta\right)(\|u\|')^2_{\frac{2n}{n-2}}<\|u\|^2_{\frac{2n}{n-2}}<\exp\left(\frac{n}{2}\delta\right)(\|u\|')^2_{\frac{2n}{n-2}}
\end{equation} 
 and 
\begin{equation}\label{forms}
\exp\left(-\left(\frac{n}{2}+1\right)\delta\right)(\|\omega\|')^2<\|\omega\|^2<\exp\left(\left(\frac{n}{2}+1\right)\delta\right)(\|\omega\|')^2,
\end{equation} 
 where $\|\cdot\|'$ is the norm on $g'$, $\|\cdot\|$ is the norm on $g$, $u$ is a a smooth function on $M$ satisfying  $\mbox{supp}(u)\subset U$ and $\omega$  is  a smooth differential form on $M$ which satisfies  $\mbox{supp}(\omega)\subset U$. By continuity of $R_g$ and $H_g$ we can assume that  
\begin{eqnarray}\label{ineqscal}
\exp\left(-\left(\frac{n}{2}+1\right)\delta\right)\int_{M}R_{g'}u^2dv_{g'}&<&\int_{M}R_{g}u^2dv_{g}\nonumber\\
&<&\exp\left(\left(\frac{n}{2}+1\right)\delta\right) \int_{M}R_{g'}u^2dv_{g'}
\end{eqnarray}and 
\begin{eqnarray}\label{ineqmean}
\exp\left(-\left(\frac{n}{2}+1\right)\delta\right)\int_{\Sigma}H_{g'}u^2da_{g'}&<&\int_{\Sigma}H_{g}u^2da_{g}\nonumber\\
&<&\exp\left(\left(\frac{n}{2}+1\right)\delta\right) \int_{\Sigma}H_{g'}u^2da_{g'}
\end{eqnarray}
 
By using partition of the unity, \eqref{function}, \eqref{forms}, \eqref{ineqscal} and \eqref{ineqmean}, we obtain, for every non-zero function $u\in C^{\infty}(M)$, that
\begin{align}\label{EQ}
\exp(-(n+1)\delta)\frac{E_{g'}(u)}{(\int_Mu^{\frac{2n}{n-2}}dv_{g'})^{\frac{n-2}{n}}} &<\frac{E_{g}(u)}{(\int_Mu^{\frac{2n}{n-2}}dv_{g})^{\frac{n-2}{n}}}\nonumber\\
 &<\exp((n+1)\delta)\frac{E_{g'}(u)}{(\int_Mu^{\frac{2n}{n-2}}dv_{g'})^{\frac{n-2}{n}}},\nonumber
\end{align} 
 where $E_g(u)$ is defined in \eqref{eq018}.
 Hence, 
 $$
\exp(-(n+1)\delta)<\frac{Y_1(M,[g'])}{Y_1(M,[g])}<\exp((n+1)\delta).
 $$
Namely, if $ g_i$ converges to $g$ in $C^0,$ $R_{g_i}$ and $H_{g_i}$ converge to $R_g$ and $ H_g$ in $C^0,$ respectively, then the ratio $\frac{Y_1(M,[g_i])}{Y_1(M,[g])}$ is close to $1.$  In particular, for every positive number $\delta>0$,  it holds in a  $\delta$-neighborhood of $g$ that 
$$
\left|Y_1(M,[g_i])-Y_1(M,[g])\right|\leq (\exp((n+1)\delta)-1)Y_1(M,[g]).
$$
 This  implies the continuity of $g\mapsto Y_1(M,[g]).$
 \end{proof}

%%%%%%%%%%%%%%%%%%%%%%%%%%%%%%%%%%%%%%%

%%%%%%REFERENCES

%%%%%%%%%%%%%%%%%%%%%%%%%%%%%%%%%%%

\bibliography{References.bib}
\bibliographystyle{acm}

\end{document}